\documentclass[12pt,draft]{amsart}
\usepackage{amsmath}
\usepackage{amssymb}
\usepackage{amsfonts}
\usepackage{mathrsfs}

\usepackage{amsthm,url,graphicx}
\usepackage{cite}

\newtheorem{theorem}{Theorem}
\newtheorem{lemma}[theorem]{Lemma}
\newtheorem{corollary}[theorem]{Corollary}
\newtheorem{proposition}[theorem]{Proposition}
\theoremstyle{definition}
\newtheorem{definition}[theorem]{Definition}
\newtheorem{remark}[theorem]{Remark}
\newtheorem*{ackno}{Acknowledgement}
\numberwithin{equation}{section}
\numberwithin{theorem}{section}

\newcommand{\N}{\mathbb{N}}
\newcommand{\G}{\mathcal{G}}
\newcommand{\Laplace}{\triangle}
\newcommand{\BigOh}{\mathcal{O}}
\renewcommand{\Re}{\mathrm{Re}}
\renewcommand{\theta}{\vartheta}
\newcommand{\trans}{\mathsf{T}}
\newcommand{\nt}{\tilde{n}}
\newcommand{\K}{\mathcal{K}}
\newcommand{\FF}{\mathbb{F}}
\newcommand{\X}{\mathscr{X}}

\DeclareMathOperator{\HyperF}{F}

\newcommand{\Hypergeom}[5]{{\sideset{_#1}{_#2}\HyperF\!\left(
      \genfrac{}{}{0pt}{0}{#3}{#4}
      \middle|\,#5\right)}}

\DeclareMathOperator{\tr}{tr}

\begin{document}
\author[Anderson]{Austin Anderson}
\address[AA]{Department of Mathematics\\
Florida State University}
\email{ana17b@my.fsu.edu}

\author[Dostert]{Maria Dostert\textsuperscript{\ddag}}
\address[MD]{Department of Mathematics\\
  KTH Royal Institute of Technology, Stockholm, Sweden}
\email{dostert@kth.se}

\author[Grabner]{Peter J. Grabner\textsuperscript{*}}
\address[PG\&RM]{Institute of Analysis and Number Theory\\
  Graz University of Technology, Graz, Austria}
\email{peter.grabner@tugraz.at}

\author[Matzke]{Ryan W. Matzke\textsuperscript{*}}
\email{matzke@math.tugraz.at}

\author[Stepaniuk]{Tetiana A. Stepaniuk\textsuperscript{$\diamond$}}
\address[TS]{Institute of Mathematics\\University of L\"ubeck, Germany; \\
Institute of Mathematics  NAS of Ukraine}
\email{stepaniuk@math.uni-luebeck.de}

\thanks{\textsuperscript{*}These authors are supported by the Austrian Science
  Fund FWF project F5503 part of the Special Research Program (SFB)
  ``Quasi-Monte Carlo Methods: Theory and Applications''.}
\thanks{\textsuperscript{\ddag}This author was partially
  supported by the Wallenberg AI, Autonomous Systems and Software Program
  (WASP) funded by the Knut and Alice Wallenberg Foundation.}
\thanks{\textsuperscript{$\diamond$}This author was
  supported by the Alexander von Humboldt Foundation.}  

\title[Riesz and Green energy on projective spaces]
{Riesz and Green energy on projective spaces}
\date{\today}
\begin{abstract}
  In this paper we study Riesz, Green and logarithmic energy on two-point
  homogeneous spaces. More precisely we consider the real, the complex, the
  quaternionic and the Cayley projective spaces. For each of these spaces we
  provide upper estimates for the mentioned energies using determinantal point
  processes. Moreover, we determine lower bounds for these energies of the same
  order of magnitude.
\end{abstract}
\maketitle
\section{Introduction}\label{sec:introduction}

Motivated by classical potential theory (see, for instance
\cite{Landkof1972:foundations_modern_potential}) discrete energies of point
sets on manifolds have been studied. More precisely, for a given symmetric and
lower semi-continuous kernel $K:\Omega\times\Omega\to\mathbb{R}$ on a metric
space $\Omega$ the \emph{discrete energy} of a set
$\omega_N=\{x_1,\ldots,x_N\}\subset\Omega$ is given by
\begin{equation*}
  E_K(\omega_N)=\sum_{\substack{i,j=1\\i\neq j}}^NK(x_i,x_j).
\end{equation*}
In rather general settings the empirical measures associated to minimizing
configurations of $E_K(\omega_N)$ for $N\to\infty$ converge weakly to the
minimizing measure of the continuous energy
\begin{equation*}
  I_K(\mu)=\iint\limits_{\Omega\times\Omega}K(x,y)\,d\mu(x)\,d\mu(y)
\end{equation*}
amongst all Borel probability measures. For more details and a comprehensive
introduction to the subject we refer to
\cite{BorodachovHardinSaff2019:EnergyRectifiableSets}.

For a sufficiently repulsive potential, one expects the minimizing
configurations of the discrete energy to be well-distributed, in some
sense. Perhaps the best-known example of such potentials are the classical
Riesz $s$-energies ($s>0$) for infinite compact $\Omega \subseteq \mathbb{R}^d$
\begin{equation*}
J_s(x,y) := \frac{1}{\| x-y\|^s}.
\end{equation*}
It has been shown in \cite{Hardin_Saff2005:minimal_riesz_energy} that, under
rather general conditions, if $s \geq \dim(\Omega)$, then the minimizers of
$E_{J_s}$ are uniformly distributed. This uniformity of minimizers does not
hold in general for $s < \dim(\Omega)$. However, due to the highly symmetric
structure of the sphere $\mathbb{S}^{d-1}$, one finds that for the Riesz
potentials $J_s$ with $0< s< d-1$ and the logarithmic potential
\begin{equation*}
J_0(x,y) = - \log( \| x-y\|),
\end{equation*}
which is obtained by a limiting process for $s\to 0$, the continuous energies
$I_{J_s}$ and $I_{J_0}$ are uniquely minimized by the uniform measure on the
sphere $\sigma$, and the minimizers of the discrete energies are uniformly
distributed.

The minimal energy of $N$ points for the kernel $J_s$ ($s \geq 0$) on a space $\Omega$,
\begin{equation*}
\mathcal{E}_{J_s}(\Omega, N) = \min_{\omega_N \subset \Omega} E_{J_s}(\omega_N)
\end{equation*}
has been investigated especially for the sphere $\mathbb{S}^{d-1}$, see for
instance \cite{Wagner1992:means_distances_surface,
  Wagner1990:means_distances_surface}.  For $0<s<d-1$ it satisfies
\begin{equation*}
  -C_1 N^{1+\frac s{d-1}}\leq \mathcal{E}_{J_s}(\mathbb{S}^{d-1}, N)-
  I_{J_s}(\sigma)N^2
  \leq -C_2 N^{1+\frac s{d-1}},
\end{equation*}
where $C_1$ and $C_2$ are positive constants.  The term $I_{J_s}(\sigma)N^2$
of highest order reflects the fact that the empirical measures of the discrete
minimizers weakly tend to $\sigma$. It is conjectured that a more precise
asymptotic equation
\begin{equation*}
  \mathcal{E}_{J_s}(\mathbb{S}^{d-1}, N)= I_{J_s}(\sigma)N^2-
  C N^{1+\frac s{d-1}}+o(N^{1+\frac s{d-1}})
\end{equation*}
holds, where the precise value of the constant $C$ is believed to reflect the
local structure of minimizing configuration. For more details we refer to
\cite{Brauchart_Hardin_Saff2012:next_order_term} and
\cite{Hardin_Saff2005:minimal_riesz_energy}. The conjectural values of the
constant are related to zeta functions of certain lattices, which relates the
question to lattice energies on Euclidean spaces.

Motivated by these results, as well as certain other recent works mentioned
below, we extend these results known for spheres to the projective spaces
$\mathbb{FP}^{d-1}$ over scalar domains $\mathbb{F}$ (the real or complex
numbers, the quaternions, or the octonions). These spaces together with the
spheres are the only compact connected two-point homogeneous spaces (see
\cite{Wang1952:Two_Point_Compact}).  On these projective spaces, we study the
energies given by the \textit{chordal} Riesz $s$-kernels
\begin{equation*}
  K_s(x,y) = \frac1{\rho(x,y)^s} =
  \frac1{\sin(\theta(x,y))^s}\quad\text{for }s>0,
\end{equation*}
and chordal logarithmic kernel
\begin{equation*}
K_0(x,y) =  - \log(\rho(x,y)) = - \log \sin(\theta(x,y))
\end{equation*}
where $\rho$ and $\vartheta$ are the chordal and geodesic metrics,
respectively, discussed below. The case of $s<\dim(\Omega)$ is the subject of
classical potential theory (see, for instance
\cite{Landkof1972:foundations_modern_potential}), the corresponding kernels are
called \emph{singular}, whereas the kernels for $s\geq\dim(\Omega)$ are called
\emph{hypersingular}. Each of the projective spaces can be embedded in a
sufficiently high dimensional unit sphere, in which case the chordal distance
becomes the Euclidean distance, making these energies the natural
generalization of the classical Riesz and logarithmic energies on the sphere,
but without requiring the embedding itself.

The study of Riesz energies on projective spaces, particularly $\mathbb{CP}^{d-1}$, has been a subject of recent interest. In\cite{Alehyane_Asserda_Assila2022:applications_projective_logarithmic,
  Asserda_Assila_Zeriahi2020:projective_logarithmic_potentials,
  Assila2018:projective_logarithmic_potentials}, the authors studied various
potential theoretic properties of the logarithmic energy on complex projective
spaces. The expected Riesz and logarithmic energies of zero sets of independent Gaussian polynomials on $\mathbb{CP}^{d-1}$ (and more generally on K\"{a}hler manifolds) was determined in \cite{Feng_Zelditch2013:random_riesz_kahler}, whereas in \cite{Beltran_Etayo2018:projective_ensemble_distribution}, the
authors computed the expected energies for certain determinantal point processes to find asymptotic upper bounds on
the Riesz and logarithmic energies. More qualitative properties of the minimizers of the logarithmic and Riesz energies on the real and complex projective spaces were studied in
\cite{Chen_Hardin_Saff2021:search_tight_frames}. The authors
found that the minimizers of these energies are uniformly distributed,
including the hypersingular case, and that these minimizers approximate tight
frames (acting as element on the real and complex spheres). Moreover, as
$s \rightarrow \infty$, the minimizers of the Riesz energies approximate best
packings on these spaces (i.e. frames with low coherence). Both tight frames on
real and complex spaces as well as best packings on Grassmannians have
applications to signal processing (see, e.g.,
\cite{Merda_Davison2014:Flexible_Codebook, Holmes_Paulsen2004:Frames_Erasure,
  Casazza_Kovecevic2001:Uniform_Tight_Frames}).

Another natural kernel to study on the projective spaces is the Green function,
$G(x,y)$, associated to the Laplace-Beltrami operator. The Green function is a
smooth potential, intrinsic to any Riemannian manifold, that behaves similarly
to a Riesz energy at short ranges. On the sphere and projective spaces, it is
in fact a function of distance only, making computations much more
feasible. The minimization of Green energies on compact Riemannian manifolds
was first studied in \cite{Beltran_Corral_CriadodelRey2019:GreenEnergy}, where
it was shown that the continuous Green energy is uniquely minimized by the
uniform measure and that minimizers of the discrete Green energy are uniformly
distributed. The minimizers of the Green energy have more recently been shown
to be well-seperated in \cite{Criado2019:separation_distance_minimal}, and to
have the optimal asymptotic quadratic Wasserstein distance from the uniform
measure in \cite{Steinerberger2021:wasserstein_inequality_minimal}. Upper
bounds for the minimal discrete Green energy on complex projective spaces were
determined in \cite{Beltran_Etayo2018:projective_ensemble_distribution} using
determinantal point processes different from the ones used in our paper. This
upper bound was of the optimal order, which was determined for general compact
Riemannian manifolds in
\cite{Steinerberger2021:wasserstein_inequality_minimal}.

\subsection{Summary of Paper and Main Results}\label{subsec:IntoResults}

In Sections \ref{sec:notation} and \ref{sec:jacobi-polynomials} we list some
necessary notation and properties of Jacobi polynomials, which we make
extensive use of in this paper.

In order to make this paper (mostly) self-contained and gather the
required material, in Section \ref{sec:harm-analys-two}, we cover the necessary
background for harmonic analysis on compact connected two-point homogeneous
spaces, as well as some of its consequences.  In Section
\ref{sec:green-function}, we obtain explicit formulae for the Green functions on
the projective spaces. In Section \ref{sec:MinimizersOfEnergies}, we show that
the Riesz and logarithmic kernels are strictly positive definite, and obtain
the following result:
\begin{theorem}\label{thm:UniformMinEnergiesIntro}
  The continuous logarithmic energy $I_{K_0}$, Green energy $I_G$, and Riesz
  $s$-energies $I_{K_s}$, for $0 < s < \dim(\mathbb{FP}^{d-1})$, are uniquely
  minimized by the uniform measure $\sigma$.
  
Moreover, if $\{ \omega_N\}_{N=2}^{\infty}$ is a sequence of minimizers for the
discrete energies $E_{K_0}$, $E_G$, or $E_{K_s}$, for
$0 < s < \dim(\mathbb{FP}^{d-1})$, then $\{ \omega_N\}_{N=2}^{\infty}$ is a
sequence of uniformly distributed point configurations.
\end{theorem}
We finish this section by discussing the properties of the heat kernel we use
to find lower bounds on the minimal discrete Green energy.

In Section~\ref{sec:rotat-invar-determ}, we define determinantal point
processes given by rotation invariant kernels on the projective spaces. The
processes are defined by projections to spaces of harmonic functions, thus they
are called \emph{harmonic ensembles} following
\cite{Beltran_Marzo_Ortega-Cerda2016:energy_discrepancy_rotationally}.

In Section~\ref{sec:bounds-green-riesz}, we provide lower and upper asymptotic bounds on the minimal discrete Riesz, logarithmic, and Green energies on the projective spaces, cumulating in the following three results:
\begin{theorem}\label{thm:BoundsMinRieszEnergy}
  For each projective space $\mathbb{FP}^{d-1}$ and
  $0 < \!s\! <\! \dim(\mathbb{FP}^{d-1}\!)$ there exist positive constants $C_s, C_s'$
  such that for $N \geq 2$
\begin{equation*}
-C_s N^{1 +\frac{s}{\dim(\mathbb{FP}^{d-1})}} \leq \mathcal{E}_{K_s}(\mathbb{FP}^{d-1}, N) - I_{K_s}(\sigma) N^2 \leq -C_s' N^{1 +\frac{s}{\dim(\mathbb{FP}^{d-1})}}.
\end{equation*}
\end{theorem}

\begin{theorem}\label{thm:BoundsMinLogEnergy}
For each projective space $\mathbb{FP}^{d-1}$, there exist positive constants $C_0, C_0'$ such that for $N \geq 2$
\begin{equation*}
-C_0 N \log(N) \leq \mathcal{E}_{K_0}(\mathbb{FP}^{d-1}, N) - I_{K_0}(\sigma) N^2 \leq -C_0' N \log(N).
\end{equation*}
\end{theorem}

\begin{theorem}\label{thm:BoundsMinGreenEnergy}
For each projective space $\mathbb{FP}^{d-1}$, there exist positive constants $C_G, C_G'$ such that for $N \geq 2$
\begin{equation*}
-C_G N^{2 -\frac{2}{\dim(\mathbb{FP}^{d-1})}} \leq \mathcal{E}_{G}(\mathbb{FP}^{d-1}, N) \leq -C_G' N^{2 -\frac{2}{\dim(\mathbb{FP}^{d-1})}},
\end{equation*}
unless $\mathbb{FP}^{d-1}= \mathbb{RP}^2$, in which case
\begin{equation*}
-C_G N\log(N) \leq \mathcal{E}_{G}(\mathbb{FP}^{d-1}, N) \leq -C_G' N \log(N).
\end{equation*}
\end{theorem}

The order of the upper bounds for each of these results is proved in Section
\ref{sec:upper-bounds-jittered} through jittered sampling using equal area
partitions with some extra control on the diameters. Such partitions exist on
general Ahlfors regular metric measure spaces by
\cite{Gigante_Leopardi2017:diameter_bounded_equal}.  However, without a deeper
understanding of the geometry of these spaces, the method of jittered sampling
does not give explicit values for the constants $C_s'$, $C_0'$, and $C_G'$ in
the above theorems. In Section~\ref{sec:upper-bounds-riesz}, we compute the
expected Riesz, logarithmic, and Green energies of the harmonic ensemble. This
provides a more concrete upper bound on the minimal energies, with an
explicit constant for the next-order term and the order of the error term,
though only for certain values of $N$ (see \eqref{eq:Riesz-expected-N},
\eqref{eq:energy-log-N}, \eqref{eq:green-expected-N}, and
\eqref{eq:RP2green-expected-N}). In addition, we compute the expected Riesz
$s$-energy for $s=\dim\mathbb{FP}^{d-1}$ of this ensemble, resulting in an
asymptotic upper bound on the minimum of this hypersingular energy. In
Section~\ref{subsec:lowerbounds}, we determine the order of the lower bounds
for the Riesz and logarithmic energies through linear programming using the
complete monotonicity of the corresponding kernels as a function of chordal
distance. Finally, in Section~\ref{sec:lower-bounds-green} we give lower bounds
for the Green energy, with explicit values of the next-order term. We achieve
these lower bounds again via linear programming, this time making use of lower
bounds for the Green function obtained from the positivity of the heat kernel.

We collect our explicit upper and lower bounds for the minimal discrete Green
energies on projective spaces, with the lower bounds holding for all $N$ and
the upper bound holding only for certain values of $N$, in Table
\ref{table:LowerAndUpperBounds} below.

\begin{table}[h]
  \renewcommand{\arraystretch}{2}
   \begin{tabular}{p{1.2cm}|p{5.216cm}|p{5.216cm}}
$\Omega$ & lower bound & upper bound\\\hline
$\mathbb{RP}^{2}$ & $- \frac{1}{2} N \log(N) +\mathcal{O}(N)$ & $- \frac{1}{2} N \log(N) +\mathcal{O}(N)$ \\\hline
$\mathbb{RP}^{3}$ & $- \frac{3}{4}\left( \pi \right)^{\frac{1}{3}} N^{2 -\frac{2}{3}}+\mathcal{O}(N \log(N))$ & $- \frac{9}{16}\left(\frac{4}{3}\right)^{\frac{2}{3}} N^{2 -\frac{2}{3}}+\mathcal{O}(N)$ \\\hline
$\mathbb{RP}^{d-1}$ & $- \frac{d-1}{4(d-3)}\left(\frac{\sqrt{\pi}}{\Gamma(d/2)}
\right)^{\frac{2}{d-1}} N^{ 2 -\frac{2}{d-1}}$ $+\mathcal{O}(N^{2-\frac{3}{d-1}})$
& $-
\frac{(d-1)^2}{8(d-2)(d-3)}\left(\frac{\sqrt{\pi}}{\Gamma(\frac{d}{2})\Gamma(\frac{d+1}{2})}\right)^{\frac{2}{d-1}}$ $\times
N^{2-\frac{2}{d-1}}+
\mathcal{O}(N^{2-\frac{3}{d-1}})$ \\\hline
$\mathbb{CP}^{d-1}$ & $-
\frac{d-1}{4(d-2)}\big(\frac{1}{\Gamma(d)}\big)^{\frac{1}{d-1}}
N^{2-\frac{2}{2d-2}}$ $+ \mathcal{O}(N^{2-\frac{3}{2d-2}})$
& $- \frac{(d-1)}{4(2d-3)}\left(\frac{1}{\Gamma(d)}\right)^{\frac{2}{d-1}}
 N^{2-\frac{2}{2d-2}}$ $+\mathcal{O}(N^{2-\frac{3}{2d-2}})$ \\\hline
$\mathbb{HP}^{d-1}$ & $-
\frac{d-1}{2(2d-3)}\big(\frac{1}{\Gamma(2d)}\big)^{\frac{1}{2d-2}} N^{2-
  \frac{2}{4d-4}} $ $+ \mathcal{O}(N^{2-\frac{3}{4d-4}})$ & $-
\frac{(d-1)^2}{(2d-3)(4d-5)}\!\left(\!\frac{1}{\Gamma(2d)
    \Gamma(2d-1)}\!\right)^{\frac{1}{2d-2}} $ $\times N^{2-\frac{2}{4d-4}}+ \mathcal{O}(N^{2-\frac{3}{4d-4}})$\\\hline
$\mathbb{OP}^{2}$ & $-\frac{2}{7} \big(\frac{6}{11!}\big)^{\frac{1}{8}}
N^{2-\frac{2}{16}} + \mathcal{O}(N^{2-\frac{3}{16}})$  & $-
\frac{16}{21}\left(\frac{6}{(11!) (8!)}\right)^{\frac{1}{8}} N^{2-\frac{2}{16}}
$ $+\mathcal{O}(N^{2-\frac{3}{16}})$ 
\end{tabular}
\caption{Lower and upper bounds for the minimal discrete Green energy on projective spaces in terms of the number of points $N$. The upper bound only holds for
  $N=\frac{(\alpha+\beta+2)_n(\alpha+2)_n}{(\beta+2)_nn!}$ (defined in \ref{sec:class-two-point})}\label{table:LowerAndUpperBounds}
\end{table}

For the complex projective spaces $\mathbb{CP}^{d-1}$ with $d > 4$, this upper
bound is an improvement upon the previously best known upper bound
\begin{equation*}
\mathcal{E}_G(N) \leq - \frac{d-1}{4(d-2)} \Big(\frac{1}{(d-1)!}\Big)^{\frac{1}{d-1}} N^{2 - \frac{2}{2d-2}}
\end{equation*}
for $N= \binom{d+n-1}{n}$ \cite[Theorem
1.3]{Beltran_Etayo2018:projective_ensemble_distribution}. We note that in their
paper, Beltr\'an and Etayo took the volume of the $\mathbb{CP}^{d-1}$ to be
$\frac{\pi^{d-1}}{(d-1)!}$, so we have adjusted their result to match our
normalization (volume being $1$). This upper bound was achieved through a
determinantal point process on the complex space $\mathbb{C}^{d-1}$ with a
kernel constructed from functions on this space, and then mapping this process
to the complex projective space $\mathbb{CP}^{d-1}$ via the map
$z \mapsto (1,z)$, creating the \textit{projective ensemble}. As pointed out,
our upper bound also comes from a determinantal point process, but the kernel
in this instance is built directly from functions on $\mathbb{CP}^{d-1}$. The
rotational invariance that results from this seems to lead to an improvement,
however, it also means that our bound holds for different values of $N$
($N= \frac{((d-1)!)^2}{n+1} \binom{d+n-1}{n}^2$).

Similarly, our estimates for the minimal Riesz and logarithmic energies on complex projective spaces resulting from the harmonic ensemble (\eqref{eq:Riesz-expected-N} and \eqref{eq:energy-log-N}) generally match or improve upon previously known results. A first instance of applying point processes to obtain estimates for the Riesz and logarithmic energies on $\mathbb{CP}^{d-1}$ is \cite{Feng_Zelditch2013:random_riesz_kahler}. There, Feng and Zelditch studied the expectation of these energies for the zero set of $d-1$ degree $m$ Gaussian random polynomials. They obtain the correct main term and the correct order of the second asymptotic term for $0\leq s\leq\min(4,D)$; nevertheless, their second order term becomes positive for $s$ close to $4$.

More recently, Beltr\'{a}n and Etayo also applied their projective ensemble to find estimates for the minimal Riesz and logarithmic energies on $\mathbb{CP}^{d-1}$. For $0 < s < 2d-2$ and $N= \binom{d+n-1}{n}$, they obtained the upper bound \cite[Theorem 3.3]{Beltran_Etayo2018:projective_ensemble_distribution})
\begin{equation*}
\mathcal{E}_{K_s}(N) \leq I_{K_s}(\sigma)N^2 -\frac{(d-1) \Gamma\left(d-1-\frac{s}{2}\right)}{\left(\Gamma(d)\right)^{1-\frac{s}{2d-2}}}N^{1+\frac{s}{2d-2}} + o(N^{1+ \frac{s}{2d-2}}).
\end{equation*}
Numerical evidence strongly suggests that the coefficient of the next-order term in this bound and the one we obtain in Theorem \ref{thm:Riesz-expected} equate at some unique value $s = s_{d-1}$ in each dimension, with their bound winning for values below $s_{d-1}$ and ours being better for values above $s_{d-1}$. These values $s_{d-1}$ appear to be bounded from above by some constant $s^{*}$. Assuming such an $s^{*}$ exists, we believe $s^{*} \approx 6.0365$, which we obtain by solving
\begin{equation*}
e^{-\frac{s^{*}}{2}}\frac{\Gamma \left(1+s^{*}\right)}{\Gamma \left(1+\frac{s^{*}}{2}\right)^2}=1,
\end{equation*}
which results from letting $d$ to tend to infinity in the ratio of the two coefficients and assuming them to be equal. In particular, our bound appears to represent an improvement for all $s$ exceeding a fixed value  independent of dimension.

We also provide the expected logarithmic energy of our harmonic ensemble on $\mathbb{CP}^{d-1}$ in Theorem \ref{thm:log-energy-upper}, achieving the same next order term, $- \frac{1}{2d-2} N \log(N)$, as achieved by the expected logarithmic energies in \cite[Corollary 3.4]
{Beltran_Etayo2018:projective_ensemble_distribution} and \cite[Corollary 1]{Feng_Zelditch2013:random_riesz_kahler}. The lower bound we determine in Theorem \ref{thm:lower-riesz} shows that this is indeed the best possible coefficient for the second order term.

\subsection{Notation}\label{sec:notation}
Throughout the paper we will use the following notations:
\begin{itemize}
\item The \emph{Pochhammer symbol}
  \begin{equation*}
    (x)_k=x(x+1)\cdots(x+k-1)=\frac{\Gamma(x+k)}{\Gamma(x)},
  \end{equation*}
\item the \emph{digamma function}
  \begin{equation*}
    \psi(x)=\frac{\Gamma'(x)}{\Gamma(x)}=-\gamma-\frac1x+
    \sum_{n=1}^\infty\left(\frac1n-\frac1{n+x}\right),
  \end{equation*}
\item the \emph{harmonic numbers}
  \begin{equation*}
    H_k=\sum_{\ell=1}^k\frac1\ell=\psi(k+1)+\gamma,
  \end{equation*}
\item where
  \begin{equation*}
    \gamma=-\Gamma'(1)=\lim_{k\to\infty}H_k-\log(k)
  \end{equation*}
  denotes the \emph{Euler-Mascheroni constant}.
\item We will also make frequent use of the asymptotic relations
\begin{equation*}
  \frac{\Gamma(n+x)}{\Gamma(n+y)}=n^{x-y}
  \left(1+\BigOh\left(\frac1n\right)\right)\text{as }n\to\infty
\end{equation*}
and
\begin{equation*}
  \binom{n+x}n=
  \frac{n^x}{\Gamma(x+1)}\left(1+\BigOh\left(\frac1n\right)\right)
  \text{as }n\to\infty.
\end{equation*}
\item We will denote the set of finite Borel measures on a space $\Omega$ as
  $\mathcal{B}(\Omega)$, the set of Borel probability measures as
  $\mathbb{P}(\Omega)$, the set of finite signed Borel measures as
  $\mathcal{M}(\Omega)$, and the set of finite signed Borel measures with total
  mass zero, i.e. $\nu \in \mathcal{M}(\Omega)$ satisfying $\nu(\Omega) = 0$,
  as $\mathcal{Z}(\Omega)$.
\end{itemize}
\subsection{Jacobi polynomials}\label{sec:jacobi-polynomials}
The classical Jacobi polynomials will play a prominent role in this paper. Thus
we collect some basic facts about them. The \emph{Jacobi polynomials}
$P_n^{(\alpha,\beta)}(t)$ are the orthogonal polynomials for the weight
function $(1-t)^\alpha(1+t)^\beta$ on the interval $[-1,1]$. Throughout the
paper we will use the substitution $t=\cos(2\theta)$. The measure is then
normalised and transformed to a measure on the interval $[0,\frac\pi2]$ that we
denote by
\begin{equation}\label{eq:nualpha}
  d\nu^{(\alpha,\beta)}(\theta)=
  \frac1{\gamma_{\alpha,\beta}}\sin(\theta)^{2\alpha+1}\cos(\theta)^{2\beta+1}\,
  d\theta,
\end{equation}
where
\begin{equation*}
  \gamma_{\alpha,\beta}=\frac{\Gamma(\alpha+1)\Gamma(\beta+1)}
  {2\Gamma(\alpha+\beta+2)}.
\end{equation*}
The Jacobi polynomials can be given
by Rodrigues' formula
(see~\cite[p.~213]{Magnus_Oberhettinger_Soni1966:formulas_theorems_special})
\begin{equation*}
  P_n^{(\alpha,\beta)}(t)=\frac{(-1)^n}{2^n n!}\frac1{(1-t)^\alpha(1+t)^\beta}
  \frac{d^n}{dt^n}(1-t)^{n+\alpha}(1+t)^{n+\beta}.
\end{equation*}
The value
\begin{equation*}
  P_n^{(\alpha,\beta)}(1)=\binom{n+\alpha}n
\end{equation*}
and the relation
\begin{equation*}
    \int_0^{\frac\pi2}\left(P_n^{(\alpha,\beta)}(\cos(2\theta))\right)^2\,
    d\nu^{(\alpha,\beta)}(\theta)
    = \frac{\alpha+\beta+1}{2n+\alpha+\beta+1}
    \frac{(\alpha+1)_n(\beta+1)_n}{n!(\alpha+\beta+1)_n}
\end{equation*}
will occur frequently throughout.

We will use the summation formula 
\begin{equation}
  \label{eq:Jacobi-sum}
  \sum_{k=0}^n\frac{2k+\alpha+\beta+1}{\alpha+\beta+1}
  \frac{(\alpha+\beta+1)_k}{(\beta+1)_k}P_k^{(\alpha,\beta)}(t)=
  \frac{(\alpha+\beta+2)_n}{(\beta+1)_n}P_n^{(\alpha+1,\beta)}(t)
\end{equation}
at several occasions; this is a special case of a connection formula for
Jacobi polynomials with different parameters given in
\cite[Theorem~7.1.3]{Andrews_Askey_Roy1999:special_functions}.

Furthermore the orthogonality of the Jacobi polynomials allows expanding
functions $F(\cos(2\theta))\in L^2([0,\frac\pi2],d\nu^{(\alpha, \beta)})$ in
terms of $P_n^{(\alpha, \beta)}$:
\begin{equation}\label{eq:JacobiExpansion}
  F(t)=\sum_{n=0}^\infty \widehat{F}(n) P_n^{(\alpha, \beta)}(t),
\end{equation}
where
\begin{equation}\label{eq:JacobiCoeff}
\widehat{F}(n)=\frac{m_n}{\left(P_n^{(\alpha,\beta)}(1)\right)^2}
\int\limits_{0}^{\frac{\pi}{2}} F(\cos(2\theta))
  P_n^{(\alpha, \beta)}(\cos(2\theta))\,d\nu^{(\alpha, \beta)}(\theta)
\end{equation}
and
\begin{equation*}
  m_n=\frac{2n+\alpha+\beta+1}{\alpha+\beta+1}
\frac{(\alpha+\beta+1)_n(\alpha+1)_n}{n!(\beta+1)_n}.
\end{equation*}
The convergence of the series \eqref{eq:JacobiExpansion} is \emph{a priori} in
the $L^2$-sense. In the case that $F$ is continuous on $[-1,1]$ and all the
coefficients $\widehat{F}(n)$ are non-negative, Mercer's theorem (see, for
instance \cite{Ferreira_Menegatto2009:eigenvalues_integral_operators}) ensures
absolute and uniform convergence.


\section{Harmonic analysis on two-point homogeneous
  spaces}\label{sec:harm-analys-two}

\subsection{Classification of two-point homogeneous spaces}
\label{sec:class-two-point}

We call a connected Riemannian manifold $(\Omega, p)$ \emph{homogeneous} if there is a
Lie group $G$ acting transitively on $\Omega$. 
This implies that $\Omega$ is homeomorphic to the quotient space $G/G_a$, where
$G_a : = \{ g \in G : ga = a \}$ is the stabilizer of a point $a \in
\Omega$. The choice of $a \in \Omega$ does not matter in this instance, as all
stabilizers are conjugate by transitivity. Let $\vartheta_p$ be the metric
induced by the metric tensor $p$. The metric induces a volume form, which we
normalize to obtain the normalized surface measure $\sigma$. If $G$ acts by
isometries this equals the measure induced by the Haar-measure on $G$.

If $G$ is the isometry group of the homogeneous space $\Omega$, we call
$\Omega$ \emph{two-point homogeneous} if for all
$x_1, x_2, y_1, y_2 \in \Omega$ with
$\vartheta_p(x_1, x_2) = \vartheta_p(y_1, y_2)$, there is an isometry $g \in G$
such that $ g x_i = y_i $, $ i=1,2 $. All two-point homogeneous Riemannian
manifolds have been classified
(see~\cite[Chapter~I.4]{Helgason2000:Groups_geometric_analysis}). The
noncompact spaces are the Euclidean spaces $\mathbb{R}^d$, the real, complex,
and quaternionic hyperbolic spaces, and the hyperbolic analogue of the Cayley
plane \cite{Tits1955:Two_Point_Non_Compact}.  The only compact connected
two-point homogeneous Riemannian manifolds are the real unit spheres
$ \mathbb{S}^{d-1}$, the real projective spaces $ \mathbb{RP}^{d-1} $, the
complex projective spaces $ \mathbb{CP}^{d-1} $, the quaternionic projective
spaces $ \mathbb{HP}^{d-1} $, and the Cayley projective plane $ \mathbb{OP}^2 $
(see~\cite{Wang1952:Two_Point_Compact}), with the quotient representations
given in \cite[pp.~28-29]{Wolf2007:harmonic_analysis_commutative}
\begin{equation*}
\begin{aligned}
\mathbb{S}^{d-1} & \cong \mathrm{SO}(d)/\mathrm{SO}(d-1) \\
\mathbb{RP}^{d-1}  & \cong \mathrm{O}(d)/
\Big( \mathrm{O}(d-1)\times \mathrm{O}(1) \Big),\\
\mathbb{CP}^{d-1} & \cong \mathrm{U}(d)/
\Big( \mathrm{U}(d-1)\times \mathrm{U}(1) \Big),\\
\mathbb{HP}^{d-1} & \cong \mathrm{Sp}(d)/
\Big( \mathrm{Sp}(d-1)\times \mathrm{Sp}(1) \Big), \\
\mathbb{OP}^2 & \cong F_4/\mathrm{Spin}(9).
\end{aligned}
\end{equation*}
When talking about the projective spaces in general we denote the scalar domain
by $\FF$.

Note that it suffices to consider $\mathbb{FP}^{d-1}$ for $ d > 2 $ only, as
$ \mathbb{FP}^1 $ is isomorphic to the sphere
$\mathbb{S}^{\dim_{\mathbb{R}} (\FF)} $
(see~\cite[p.~170]{Baez2002:Octonions}),  so those will not be considered in
what follows.

For each two-point homogeneous space with underlying scalar
domain $\FF$, we associate parameters
\begin{equation}
    \label{eq:JacobiAlphaBeta}
    \alpha = (d-1) \frac{\dim_{\mathbb{R}} (\FF)}{2}-1, \; \; \; \;
    \beta = \begin{cases}
        \alpha, & \text{ for } \Omega = \mathbb{S}^{d-1};\\
        \frac{\dim_{\mathbb{R}}(\FF)}{2} -1, &
        \text{ for } \Omega = \mathbb{FP}^{d-1}.
    \end{cases}
\end{equation}
The dependence of $\alpha$ and $\beta$ on the space and its dimension will be
clarified in Section~\ref{sec:irred-repr}. Furthermore, we denote
$D=\dim(\Omega)=2\alpha+2$ the dimension of the space $\Omega$ as a real
manifold.

From now on, $\Omega$ always refers to a two-point homogeneous space, equipped with metric tensor $p$ and corresponding
$G$-invariant probability measure $\sigma$, i.e. the normalized uniform surface
measure. We let $ \vartheta $ denote the geodesic distance, normalized to take
values in $ [0, \frac{\pi}{2\kappa}]$, where $\kappa=\frac12$ or $\kappa=1$, if
$\Omega$ is a sphere or projective space, respectively. We can also define a
\textit{chordal metric} $\rho$ on each of these spaces by
\begin{equation}\label{eq:Chordal Metric Def}
  \rho(x,y) = \sin(\kappa\vartheta(x,y)) =
  \sqrt{\frac{1 -\cos(2 \kappa \vartheta(x,y))}{2}},
  \quad x,y \in \Omega.
\end{equation}
Note that on the sphere, this is the Euclidean distance $ \frac12\| x - y\|$ in
ambient space, and on the complex projective space this is also known as the \textit{Fubini-Study metric}.

Each projective space $\mathbb{FP}^{d-1}$ can be canonically embedded into the
unit sphere $\mathbb{S}^{\tilde{d}-1}$, where
$\tilde{d} = d(\alpha + 2)=\frac{d(d-1)}{2}\dim_{\mathbb{R}}(\FF)+d$, so that the chordal
metric is equivalent to the Euclidean metric on this embedding, which we will
now show.

Let $\mathcal{H}(\FF^{d})$ be the set of all Hermitian $d \times d$ matrices
with entries in $\FF$. We see that $\mathcal{H}(\FF^d)$ is a linear space over
$\mathbb{R}$ of dimension $\tilde{d}$, equipped with the symmetric real-valued
inner product
\begin{equation*}
  \langle A, B \rangle = \frac{1}{2} \mathrm{tr}(A\overline{B^\trans} +
  B\overline{A^\trans}) =
  \Re(\mathrm{tr}(A\overline{B^\trans})) =
  \Re \sum_{i,j = 1}^{d} a_{i,j} \overline{b_{i,j}}
\end{equation*}
and norm
\begin{equation}\label{eq:Proj Space Norm}
\| A \|_{\mathcal{H}(\FF^{d})} =  ( \mathrm{tr}(A \overline{A^\trans}) )^{\frac{1}{2}} =\Big( \sum_{i,j=1}^{d} |a_{i,j}|^2 \Big)^{\frac{1}{2}}.
\end{equation}

For $\FF \neq \mathbb{O}$, the orthogonal projector
$\Pi_x \in \mathcal{H}(\FF^d)$ ($x \in \FF^d$, $\|x\| = 1$) onto
a one-dimensional subspace $x \FF$, can be given by the matrix
$\Pi_x = (x_i \overline{x}_j)_{1 \leq i, j \leq d}$, with
$x = (x_1, \ldots, x_d)$. Thus, the projective space can be written as
\begin{equation}\label{eq:AssociataiveProjModel}
\mathbb{FP}^{d-1} \cong \{ \Pi \in \mathcal{H}(\FF^d) : \Pi^2 = \Pi, \mathrm{Tr}(\Pi) = 1 \}.
\end{equation}
The group of isometries $U(d, \FF)$ acts on these projectors by $g(\Pi) = g \Pi g^{-1}$.

For the Cayley plane, a similar model (as well as a detailed discussion) is
given in \cite{Baez2002:Octonions, Freudenthal1953:Octonions}. In this model,
one defines the Cayley plane by
\begin{equation}\label{eq:NonassociataiveProjModel}
\mathbb{OP}^2 \cong \{ \Pi \in \mathcal{H}(\mathbb{O}^3) : \Pi^2 = \Pi, \mathrm{Tr}(\Pi) = 1 \}.
\end{equation}
Each matrix can be written as
$\Pi_x = ( x_i \overline{x_j})_{1 \leq i,j \leq 3}$, for a vector
$x = (x_1, x_2, x_3) \in \mathbb{O}^3$ with
$\|x \|^2 = |x_1|^2 + |x_2|^2 + |x_3|^2 = 1$ and $(x_1 x_2) x_3 = x_1(x_2 x_3)$
\cite[Lemma~14.90]{Harvey1990:Spinors_Cal}.

Equations \eqref{eq:Proj Space Norm}, \eqref{eq:AssociataiveProjModel}, and
\eqref{eq:NonassociataiveProjModel} show us that for any
$\Pi \in \mathbb{FP}^{d-1}$, as defined by the above models,
\begin{equation*}
\| \Pi \|_{\mathcal{H}(\FF^{d})}^2 = \mathrm{Tr}(\Pi^2) = \mathrm{Tr}(\Pi) = 1, 
\end{equation*}
so the projective spaces are submanifolds in the unit sphere
\begin{equation*}
\mathbb{FP}^{d-1} \subset \{ \Pi \in \mathcal{H}(\FF^d): \| \Pi \|_{\mathcal{H}(\FF^{d})} = 1 \} \subset \mathcal{H}(\FF^d) \cong \mathbb{R}^{\tilde{d}}.
\end{equation*}
This provides an embedding $A$ of $\mathbb{FP}^{d-1}$ into the sphere $\mathbb{S}^{\tilde{d}-1}$. The chordal metric $\rho( \Pi_1, \Pi_2)$, for $\Pi_1, \Pi_2 \in \mathbb{FP}^{d-1}$, is then defined as the Euclidean distance in the embedding
\begin{equation*}
  \begin{split}
    \rho(\Pi_1, \Pi_2) &= \sqrt{ 1 - \langle \Pi_1 , \Pi_2 \rangle } =
    \frac{1}{\sqrt{2}} \| \Pi_1 - \Pi_2\|_{\mathcal{H}(\FF^{d})} \\
    &=    \frac{1}{\sqrt{2}} \| A(\Pi_1) - A(\Pi_2) \|.
  \end{split}
\end{equation*}

\subsection{The Laplace operator and its eigenfunctions}
\label{sec:laplace-operator-eigen}

Let $\Laplace_\Omega$ be the Laplace-Beltrami operator on $\Omega$ induced by
the Riemannian metric $p$, and let
$0= \lambda_0 < \lambda_1 < \cdots$ be the eigenvalues of $\Laplace_\Omega$ and
for each $k \in \mathbb{N}$, let $V_k : = V( \Laplace_\Omega, \lambda_k)$ be
the corresponding eigenspace and $m_{k} = \dim(V_k)$ be the multiplicity of
$\lambda_k$. Notice that we follow the convention of geometry choosing the sign
of the operator so that the eigenvalues are non-negative.

We then have the following version of the Spectral Theorem
\cite[Chapter~3, Theorem~1.3, and
Remark~1.2]{Craioveanu_Puta_Rassias2001:Spectral_Geometry}

\begin{theorem}\label{thm:Eigen Decomposition L2}
  The eigenvalues of $\Laplace_\Omega$ can be arranged in increasing order
  $0=\lambda_0 < \lambda_1 < \cdots$, where
  $\lim_{k \rightarrow \infty} \lambda_k = \infty$, each eigenspace $V_k$ has
  finite dimension, and
\begin{equation*}
  L^2(\Omega, \sigma) =
  \overline{\bigoplus_{k=0}^{\infty} V_k}.
\end{equation*}
\end{theorem}

For any Riemannian manifold $M$, we can use geodesic polar coordinates
$(r, \theta_1,\ldots, \theta_{\dim_{\mathbb{R}}(M)-1})$ to parametrize a
sufficiently small neighborhood $U$ of any point $a$, giving $U$ a positive
orientation and having a radial component $r(x)$ which is the distance of
$x \in M$ from $a$, as described in
\cite[Chapter~2]{Berger_Gauduchon_Mazet1971:le_spectre_dune},
\cite[Chapter~2.4]{Rosenberg1997:Laplace_Riemannian_Manifolds}, and
\cite[Chapters~IX.5, X.4, and~X.7.4]{Helgason1962:Diff_Geo_Sym_Spaces}.

In particular, on a two-point homogeneous space $\Omega$,
such a polar coordinate parameterization can be defined on
$ \Omega_a : = \Omega \setminus \mathbb{S}_{\Omega}(a, \frac{\pi}{2\kappa})$,
where $ \mathbb{S}_{\Omega}(a,r) = \{ x \in \Omega: \vartheta(a, x) = r\}$ for
$r \in [0, \frac{\pi}{2\kappa}]$ \cite[Chapters~IX.5, X.4,
and~X.7.4]{Helgason1962:Diff_Geo_Sym_Spaces}. This allows us to separate $\Laplace_{\Omega}$ into a radial and angular component on this set.

\begin{theorem}[\!\!{\cite[Chapter~X.7.4, Lemma~7.12]%
    {Helgason1962:Diff_Geo_Sym_Spaces}}]
  \label{thm:LaplaceBeltramiRadial2}
If $f \in C^{\infty}(\Omega)$ and $a\in\Omega$, then on $\Omega_a$ the Laplace
operator can be expressed in terms of geodesic polar coordinates by
\begin{equation*}
  \Laplace_{\Omega} f = -\frac1{A(r)}\frac{\partial}{\partial r}
  \left(A(r)\frac{\partial f}{\partial r}\right)+
  \Laplace_{\theta}f,
\end{equation*}
where $\Laplace_{\theta}$ is the Laplace operator on
$\mathbb{S}_\Omega(a,r)$, and $A(r)$ denotes the surface measure of
$\mathbb{S}_\Omega(a,r)$.
\end{theorem}

For each
$r \in [0, \frac{\pi}{2\kappa}]$, $\mathbb{S}_{\Omega}(a,r)$ is a submanifold
of $\Omega$ with Riemannian structure induced by that of $\Omega$. In this case 
we have for $0 < r < \frac{\pi}{2\kappa}$ (see \cite[Proposition~5.6 and
p.~171]%
{Helgason1965:Radon_Trans_Two_Pt_Homo})
\begin{equation}\label{eq:AreaOfSphere}
A(r) = c \kappa^{-2\alpha-1} \sin^{2\alpha+1}(\kappa r) \cos^{2\beta+1}(\kappa r),
\end{equation}
where $c$ is a constant depending on the structure of $\Omega$, the values
$\alpha$ and $\beta$ are given by \eqref{eq:JacobiAlphaBeta}. For the spaces
$\mathbb{FP}^{d-1}$ we choose $\kappa=1$, whereas for $\mathbb{S}^{d-1}$ we set
$\kappa=\frac12$.

For functions $f$ only depending on $r$ the Laplace operator then becomes
\begin{equation*}
  \Laplace_r = -\frac{1}{\sin^{2\alpha+1}(\kappa r) \cos^{2\beta+1}(\kappa r)}
  \frac{d}{dr}
  \left( \sin^{2\alpha+1}(\kappa r)\cos^{2\beta+1}(\kappa r)\frac{d}{dr} \right) ,
\end{equation*}
which, making the substitution $z = \cos( 2 \kappa r)$, becomes
\begin{equation*}
  \Laplace_z = -\frac{4\kappa^2}{(1-z)^\alpha (1+z)^\beta} \frac{d}{dz}
  \left( (1-z)^{\alpha+1} (1+z)^{\beta+1}  \frac{d}{dz}\right),
\end{equation*}
\cite[pp.~177-178]{Gangolli1967:Pos_Def_Ker_Homo_Sp}. This is the Jacobi operator,
for which the only eigenfunctions continuous on $[-1,1]$ are the Jacobi
polynomials $P^{(\alpha,\beta)}_k(z)$, with corresponding eigenvalues
$\lambda_k=4\kappa^2k(k+\alpha+\beta+1)$ (see
\cite[Theorem~4.2.2]{Szegoe1975:orthogonal_polynomials}).
Summing up this discussion we have shown
\begin{theorem}\label{thm:Jacobi}
  Let $\Omega$ be a two-point homogeneous space of diameter $\frac\pi{2\kappa}$
  and $a\in\Omega$. Then the eigenfunctions of $\Laplace_{\Omega}$ on $\Omega$ depending
  only on $\theta(a,x)$ are given by
  \begin{equation*}
    cP_k^{(\alpha,\beta)}(\cos(2\kappa\theta(x,a)))\quad\text{with
    }c\in\mathbb{R}
  \end{equation*}
  and the corresponding eigenvalues are
  \begin{equation*}
\lambda_k = 4\kappa^2k (k + \alpha + \beta + 1)
\end{equation*}
with the values of $\alpha$ and $\beta$ given by \eqref{eq:JacobiAlphaBeta}.
\end{theorem}

The eigenvalues and the dimensions of their corresponding eigen\-spaces are
given in Table~\ref{tab:eigen} (see
\cite{Bettiol_Lauret_Piccione2020:full_laplace_spectrum,
  Cahn_Wolf1976:zeta_functions_their,
  Grinberg1983:spherical_harmonics_integral,
  Shatalov2001:isometric_embeddings_cubature}).
\begin{table}[h]
  \renewcommand{\arraystretch}{1.8}
  \begin{tabular}{l| l| l|l|l}
    $\Omega$  & $\alpha$&$\beta$& $\lambda_k$&$m_k = \dim(V_k)$ \\[2mm]\hline
    $\mathbb{S}^{d-1}$ &$\frac{d-3}2$&$\frac{d-3}2$& $k (k+d-2)$ &
    $ \frac{2k+d-2}{d-2}\binom{k+d-3}{d-3}$\\[2mm]\hline
    $\mathbb{RP}^{d-1}$ & $\frac{d-3}2$&$-\frac12$&$2k (2k+d-2)$ &
    $ \frac{4k+d-2}{d-2}\binom{2k+d-3}{d-3}$\\[2mm]\hline
    $\mathbb{CP}^{d-1}$ & $d-2$&$0$&$4k (k+d-1)$ &
    $ \frac{2k+d-1}{d-1}\binom{d+k-2}{d-2}^2$\\[2mm]\hline
    $\mathbb{HP}^{d-1}$&$2d-3$&$1$ & $4k (k+2d-1)$ &
    $ \frac{2k+2d-1}{(2d-1)(2d-2)} \binom{k+2d-2}{2d-2} \binom{k+2d-3}{2d-3}$
    \\[2mm]\hline
    $\mathbb{OP}^{2}$ & $7$&$3$&$4k (k+11)$ &
    $ \frac{2k+11}{1320} \binom{k+10}{7} \binom{k+7}{7}$\\
  \end{tabular}
  \caption{\label{tab:eigen}
    The eigenvalues and dimensions of the eigenspaces of the
    Laplace operator for two-point homogeneous spaces}
\end{table}

For all two-point homogeneous spaces, we have the general formula
\begin{equation*}
  m_k =\frac{2k+\alpha+\beta+1}{\alpha+\beta+1}
  \frac{(\alpha+\beta+1)_k(\alpha+1)_k}{k!(\beta+1)_k}.
\end{equation*}

\subsection{Irreducibility of representations}\label{sec:irred-repr}

Let $\Omega$ be a two-point homogeneous space, with isometry
group $G$ and $G_a$ the stabilizer of some $a \in \Omega$, meaning that
$\Omega \simeq G / G_a$. We then have an orthogonal representation of $G$ in
the pre-Hilbert space $C^{\infty}(\Omega)$ given by $g\mapsto f(gx)$. Let $V$
be a finite-dimensional subspace of $C^{\infty}(\Omega)$, invariant under this
representation, and let $q_V$ be the representation induced in $V$.

\begin{definition}\label{def:ZonalGHM}
  A function $f$ is called \emph{zonal} with respect to $a\in\Omega$
  if $f(a) \neq 0$, and for all $g \in G_a$
  and $x \in \Omega$, $f(gx) = f(x)$. The zonal functions of $V$ form a vector
  subspace of $V$, denoted by $Z_a(V)$.
\end{definition}

\begin{lemma}\label{lem:GHMIrreducibleZonal}
  If $V \neq \{0\}$, then $Z_a(V) \neq \{0\}$. If $\dim(Z_a(V)) = 1$ for one
  (and thus all) $a\in\Omega$, $q_V$ is
  irreducible.
\end{lemma}

\begin{proof}
Consider the linear map $\phi: V \rightarrow \mathbb{R}$ defined by
\begin{equation*}
\phi(f) = f(a).
\end{equation*}

Since, by assumption, $V$ contains a non-zero function, $G$ is transitive on
$\Omega$, and $V$ is $G$-invariant there exists some $f \in V$ such that
$\phi(f)$ is not zero. Thus $\ker (\phi)$ is a subspace of $V$, which is
invariant under $G_a$. The orthogonal subspace $(\ker (\phi))^{\perp}$ is also
invariant under $G_a$. Now, let $f \in (\ker (\phi))^{\perp}$ and $\mu$ be the
Haar measure on $G_a$. Define $f^*(x) = \int_{G_a} f( h x) \,d \mu(h)$
for all $x \in \Omega$. Then $f^*$ is in $V$, nonzero at $a$, and is
$G_a$-invariant, meaning it is a zonal function,
proving our first statement.

Suppose that $q_V$ is reducible. Then, since it is an orthogonal
representation, we have a decomposition
\begin{equation*}
  V = V' \oplus V''
\end{equation*} where $V'$ and $V''$ are
both $G$-invariant and nontrivial. Then each of the two spaces contains a
nontrivial subspace of zonal functions, so $\dim(Z_a(V)) \geq 2$.
\end{proof}

\begin{proposition}\label{prop:GHMIrreducibleSubspaces}
  For each $k \in \mathbb{N}_0$, let $V_k$ be the space of eigenfunctions of
  $\Laplace_{\Omega}$ associated to the $k$-th eigenvalue $\lambda_k$, and let $q_{k}$
  be the orthogonal representation of $G$ in $V_k$. The zonal functions of
  $V_k$ are exactly $c P^{(\alpha, \beta)}_{k}(\cos(2 \kappa\vartheta(x,a)))$
  for $c \in \mathbb{R}$, and so $q_{k}$ is irreducible.
\end{proposition}

\begin{proof}
  By Lemma \ref{lem:GHMIrreducibleZonal}, $Z_a(V_k) \neq \{ 0\}$, and it
  suffices to show that $\dim(Z_a(V_k)) \leq 1$. Let $f$ be a zonal
  function. Then two-point homogeneity and Theorem~\ref{thm:Jacobi} tell us that
  $f(x)=cP_k^{(\alpha,\beta)}(\cos(2\kappa\theta(x,a)))$ for some
  $c\in\mathbb{R}$. Thus the space $Z_a(V_k)$ is one-dimensional, and the
  proposition is proved.
\end{proof}

As a consequence of Theorem~\ref{thm:Eigen Decomposition L2} and
Proposition~\ref{prop:GHMIrreducibleSubspaces} we have
\begin{proposition}\label{prop:finite-invariant}
  Let $H$ be a finite-dimensional $G$-invariant subspace of
  $L^2(\Omega,\sigma)$. Then there exist $0\leq k_1<\cdots<k_m$ such that
  \begin{equation*}
    H=V_{k_1}\oplus\cdots\oplus V_{k_m},
  \end{equation*}
  where $V_k$ is the eigenspace of $\Laplace_{\Omega}$ corresponding to the eigenvalue
  $\lambda_k$.
\end{proposition}
\subsection{The addition formula for eigenfunctions of the  Laplace operator}\label{sec:addit-theor-eigenf}

For each $k \in \mathbb{N}_0$, $Y_{k,1}, \ldots, Y_{k, m_k}$
be real-valued functions and form an orthonormal basis of $V_k$, under the inner product
$\langle Y, Z \rangle = \int_{\Omega} Y(x) Z(x)\, d\sigma(x)$.

\begin{theorem}[Addition Formula]\label{thm:Addition Formula}
  For each $k \in \mathbb{N}_0$ and $x,y \in \Omega$,
  \begin{equation}
  \sum_{m=1}^{m_k} Y_{k, m}(x) Y_{k,m}(y) = \frac{m_k}{P^{(\alpha, \beta)}_{k}(1)}
  P^{(\alpha, \beta)}_{k}(\cos(2\kappa \vartheta(x,y))).\label{eq:addition}
\end{equation}

\end{theorem}

\begin{proof}
Let
\begin{equation*}
Y_k(x) := \begin{pmatrix}
Y_{k,1}(x) \\
\vdots \\
Y_{k,m_k}(x)
\end{pmatrix},
\end{equation*}
and note that $\sum_{m=1}^{m_k} Y_{k, m}(x) Y_{k,m}(y) = Y_k(x)^\trans Y_k(y)$. For
any $\gamma \in G$, $Y_{k,1}(\gamma x),\ldots, Y_{k,m_k}(\gamma x)$ is also an
orthonormal basis of $V_k$, thus there is a
$q_{k}(\gamma) \in \mathrm{O}(V_k)$, so that
\begin{align*}
  Y_k(\gamma x)^\trans Y_k( \gamma y) &=
  ( q_{k}(\gamma) Y_k(x))^\trans q_{k}(\gamma) Y_k(y)\\ 
  &=
  Y_k(x)^\trans q_{k}(\gamma)^\trans q_{k}(\gamma) Y_k(y) = Y_k(x)^\trans Y_k(y).
\end{align*}
In particular, this means that for all $\gamma \in G_y$, 
\begin{equation*}
Y_k(\gamma x)^\trans Y_k(y) = Y_k(\gamma x)^\trans Y_k(\gamma y) = Y_k(x)^\trans Y_k(y)
\end{equation*}
making $Y_k(x)^\trans Y_k(y)$ a zonal function of $x$, thus it must be a multiple
of $P^{(\alpha, \beta)}_{k}(\cos(2 \vartheta(x,y)))$: 
\begin{equation*}
  Y_k(x)^\trans Y_k(y)=c_kP_k^{(\alpha, \beta)}(\cos(2\kappa\theta(x,y))).
\end{equation*}
Setting $x=y$ and integrating yields
\begin{equation*}
  \sum_{m=1}^{m_k}\int_\Omega Y_{k,m}(x)^2\,d\sigma(x)=m_k=
  c_kP_k^{(\alpha,\beta)}(1),
\end{equation*}
which gives \eqref{eq:addition}.
\end{proof}

We note that one can achieve such an addition formula without making use of
properties of the Laplace operator and its eigenfunctions, and rather
making use of representation theory (see
\cite[Chapter~9.2.1]{Conway_Sloane1999:sphere_packings_lattices},
\cite[Section~3.2]{Levenshtein1998:Universal_Bounds_Codes}).

\begin{remark}
  From now on we only consider the projective spaces $\mathbb{RP}^{d-1}$,
  $\mathbb{CP}^{d-1}$,  $\mathbb{HP}^{d-1}$, and $\mathbb{OP}^2$; the case of
  the sphere $\mathbb{S}^{d-1}$ has been treated in a very similar manner in
  \cite{Beltran_Marzo_Ortega-Cerda2016:energy_discrepancy_rotationally}. Thus
  we set $\kappa=1$ for the remaining part of the paper.
\end{remark}

As a corollary of the area formula \eqref{eq:AreaOfSphere}, we can see that the
weighted measure $\nu^{(\alpha, \beta)}$ given by \eqref{eq:nualpha} is related
to integration on $\Omega$ in the following way: for any $a \in \Omega$, and
$F(\cos(2\theta))\in L^1([0,\frac\pi2],\nu^{\alpha,\beta})$,
\begin{equation}\label{eq:JacobiProjection}
    \int\limits_{\Omega} F( \cos(2\vartheta( x, a)))\, d \sigma(x)
    =        \int\limits_0^{\frac\pi2}F(\cos(2\theta))
    \,d\nu^{\alpha,\beta}(\theta).
\end{equation}

\subsection{The Green function}\label{sec:green-function}
The \emph{Green function} $G$ of the Laplace-Bel\-trami operator is given by the
integral operator solving the equation
\begin{equation*}
  \Laplace_{\Omega} g=f
\end{equation*}
by
\begin{equation*}
  g(x)=\int_\Omega G(x,y)f(y)\,d\sigma(y)
\end{equation*}
with the additional condition $\int_\Omega g(x)\,d\sigma(x)=0$. In the case of
two-point homogeneous spaces the bivariate function $G(x,y)$ actually only
depends on the distance between $x$ and $y$ (see
\cite{Beltran_Corral_CriadodelRey2019:GreenEnergy}). In fact, as we will show
below, the Green functions on the projective spaces have the following closed
form expressions.
\begin{proposition}\label{prop:GreenKernelFormula}
The Green function for a projective space $\Omega$ is given by
\begin{equation*}
\begin{split}
 G(x,y) = &\frac1{4(\alpha+\beta+1)} \Bigg( \sum_{\ell=1}^\alpha
  \frac{\binom \alpha\ell(\ell-1)!}{(\beta+\alpha+1-\ell)_\ell}
 \frac1{\rho(x,y)^{2\ell}} - 2\log(\rho(x,y)) \\
  &-  \psi(\alpha+\beta+1) - \psi(\alpha+\beta+2)-   \gamma  +\psi(\beta+1) \Bigg)
\end{split}
\end{equation*}
for $\alpha \in \mathbb{N}_0$ and
\begin{equation*}
  \begin{split}
 G(x,y) =  &\frac{\cos(\theta(x,y))}{2\alpha + 1}\left(\frac\pi2-\theta(x,y) \right)
    \sum_{\ell=0}^{\alpha - \frac{1}{2}} \frac{(\frac12)_\ell}{\ell!} \frac1{\rho(x,y)^{2\ell+1}}\\
& +  \frac1{4\alpha + 2}\sum_{\ell=1}^{\alpha - \frac{1}{2}} \frac1\ell\left(\frac{(\alpha + 1-\ell)_\ell}
      {(\alpha + \frac{1}{2}-\ell)_\ell}
      -\frac{(\frac12)_\ell}{\ell!}\right)\frac1{\rho(x,y)^{2\ell}}\\
    & -\frac1{2\alpha+1}H_{\alpha-\frac{1}{2}}-\frac1{(2\alpha+1)^2}
  \end{split}
\end{equation*}
for $\alpha\in\frac12+\mathbb{N}_0$.
\end{proposition}
\begin{proof}
  Using the addition formula~\eqref{eq:addition} we obtain the \emph{formal}
  expression
\begin{equation*}
  \begin{split}
    &G(x,y)\\
    = &\sum_{n=1}^\infty\frac{2n+\alpha+\beta+1}
    {4(\alpha+\beta+1)n(n+\alpha+\beta+1)}
    \frac{(\alpha+\beta+1)_n}{(\beta+1)_n}
    P_n^{(\alpha,\beta)}(\cos(2\theta(x,y))).
  \end{split}
\end{equation*}
Indeed, this series converges absolutely only for $\alpha<\frac12$,
converges conditionally for $\frac12\leq\alpha<\frac32$, and does not
converge at all for $\alpha\geq\frac32$. We will show that the series converges
in the sense of Abel's summation method (see
\cite{Zeller_Beekmann1970:theorie_limitierungsverfahren}) and we will compute
the limit.

We start with
\cite[p.~301,
eq.~(6.4.7)]{Andrews_Askey_Roy1999:special_functions} 
\begin{multline*}
\sum_{n=0}^\infty\frac{(\alpha+\beta+1)_n}{(\beta+1)_n}P_n^{(\alpha,\beta)}(t)z^n\\
  =\frac1{(1+z)^{\alpha+\beta+1}}\Hypergeom21%
{\frac{\alpha+\beta+1}2,\frac{\alpha+\beta+2}2}%
{\beta+1}{\frac{2z(1+t)}{(1+z)^2}}
\end{multline*}
and define
\begin{equation*}
  \begin{split}
    &\G_{\alpha,\beta}(\cos(2\theta))
    =\frac1\eta\lim_{z\to1}
    \sum_{n=1}^\infty\frac{2n+\eta}{4n(n+\eta)}
    \frac{(\eta)_n}{(\beta+1)_n}P_n^{(\alpha,\beta)}(\cos(2\theta))z^n\\
    =&\int\limits_0^1\!\left(\frac1{(1+z)^{\eta}}
      \Hypergeom21{\frac{\eta}2,\frac{\eta+1}2}
      {\beta+1}{\frac{4\cos(\theta)^2z}{(1+z)^2}}-1\!\right)
    \frac{1+z^{\eta}}{4\eta z}\,dz,
  \end{split}
\end{equation*}
where the convergence of the integral shows the existence of the limit
for $|\cos(\theta)|<1$. Here we have set $\eta=\alpha+\beta+1$ for short, which
we will use for this computation in the sequel.

The coefficient of $\cos(\theta)^{2n}$ for $n\geq1$ is given by
\begin{equation*}
  \frac{(\eta)_{2n}}
  {(\beta+1)_nn!}\int_0^1\frac{z^n}{(1+z)^{2n+\eta}}\frac{1+z^\eta}{4\eta z}\,dz,
\end{equation*}
where we have used
\begin{equation*}
 4^n\left(\frac\eta2\right)_n \left(\frac{\eta+1}2\right)_n=(\eta)_{2n}.
\end{equation*}
The integral is split into two parts; we first use the substitution $z=\frac1t$
to obtain
\begin{equation*}
  \int_0^1\frac{z^{n+\eta-1}}{(1+z)^{2n+\eta}}\,dz
  =\int_1^\infty\frac{t^{n-1}}{(1+t)^{2n+\eta}}\,dt.
\end{equation*}
This shows that
(see~\cite[p.~6]{Magnus_Oberhettinger_Soni1966:formulas_theorems_special})
\begin{equation*}
  \int_0^1\frac{z^n}{(1+z)^{2n+\eta}}\frac{1+z^\eta}{4\eta z}\,dz=
  \frac1{4\eta}\int_0^\infty\frac{z^{n-1}}{(1+z)^{2n+\eta}}\,dz
  =\frac{(n-1)!}{4\eta(n+\eta)_n}.
\end{equation*}
For the constant term in the Taylor expansion we obtain
\begin{equation*}
  C_\eta=-\int_0^1\left(1-\frac1{(1+z)^\eta}\right)\frac{1+z^\eta}{4\eta z}\,dz=
  -\frac1{4\eta}\left(\psi(\eta+1)+\gamma\right).
\end{equation*}
This can be obtained by splitting the integral in a similar way as above with
some slight modification to preserve convergence.

Putting everything together yields
\begin{equation*}
  \begin{split}
    \G_{\alpha,\beta}(\cos(2\theta))=
    -\frac1{4(\alpha+\beta+1)}\Big(&\gamma+\psi(\alpha+\beta+2) \\
    &+
    \sum_{n=1}^\infty\frac{(\alpha+\beta+1)_n}
    {n(\beta+1)_n}\cos(\theta)^{2n}\Big).
  \end{split}
\end{equation*}
The derivation up to now was valid for all $\alpha,\beta>-1$. For the specific
values of $\alpha$ and $\beta$ occurring in the context of projective spaces it
turns out that the functions $\G_{\alpha,\beta}$ are indeed elementary. For an
alternative derivation of the Green function on the spaces $\mathbb{FP}^{d-1}$
we refer to \cite[Appendix~A.1]{Beltran_Corral_CriadodelRey2019:GreenEnergy}

For $\alpha=k\in\N_0$ this is immediate from the observation that
\begin{equation*}
  \frac{(\alpha+\beta+1)_n}{n(\beta+1)_n}=\frac{(n+\beta+1)_k}{n(\beta+1)_k},
\end{equation*}
which is a polynomial in $n$ of degree $k$ divided by $n$. More precisely, we
have
\begin{equation*}
  (n+\beta+1)_k=\sum_{\ell=0}^k\binom k\ell (\beta+1)_{k-\ell}(n)_\ell,
\end{equation*}
from which we derive
\begin{align*}
  &\sum_{n=1}^\infty\frac{(\alpha+\beta+1)_n\cos(\theta)^{2n}}{n(\beta+1)_n}\\
  = &\frac1{(\beta+1)_k}\sum_{\ell=0}^k\binom k\ell(\beta+1)_{k-\ell}
  \sum_{n=1}^\infty\frac{(n)_{\ell}}{n}\cos(\theta)^{2n}\\
  =&  
 \sum_{\ell=1}^k\frac{\binom k\ell(\ell-1)!}{(\beta+k+1-\ell)_\ell}
  \frac1{\sin(\theta)^{2\ell}}+\log\frac1{\sin(\theta)^2}\\
  -&\left(\psi(\beta+k+1)-\psi(\beta+1)\right).
\end{align*}
For $\alpha\in\N_0$ this gives
\begin{align*}
\begin{split}
  &\G_{\alpha,\beta}(\cos(2\theta))= \frac1{4(\alpha+\beta+1)} \Bigg(\sum_{\ell=1}^\alpha\binom \alpha\ell
  \frac{(\ell-1)!}{(\beta+\alpha+1-\ell)_\ell}
 \frac1{\sin(\theta)^{2\ell}} - \gamma \\
  &+\log\frac1{\sin(\theta)^2} -\psi(\alpha+ \beta+1)-\psi(\alpha+\beta+2)-\psi(\beta+1)\Bigg).
   \end{split}
\end{align*}


The case that $\alpha\in\frac12+\N_0$ only occurs in
the case of real projective spaces, when $\beta=-\frac12$. In this case we consider the functions
\begin{align*}
  F_k(z)&=\sum_{n=1}^\infty\frac{(k+1)_n}{n(\frac12)_n}z^n\\
  G_k(z)&=2(k+1)\sum_{n=0}^\infty\frac{(k+2)_n}{(\frac32)_n}z^n
   =2(k+1)\Hypergeom21{1,k+2}{\frac32}{z};
\end{align*}
these satisfy $F_k'=G_k$. From \cite[p.~463,
eqns.~132, 133]{Prudnikov_Brychkov_Marichev1990:integrals_series_3} we infer
that
\begin{equation*}
  G_k(z)=\frac{(\frac32)_{k}}{k!}\frac{\arcsin(\sqrt z)}{\sqrt z}
  \frac1{(1-z)^{k+\frac32}}+
  \sum_{\ell=0}^k\frac{(k+\frac32-\ell)_\ell}{(k+1-\ell)_\ell}
  \frac1{(1-z)^{\ell+1}}.
\end{equation*}
From this we compute
\begin{multline*}
  F_k(z)=\frac{2\sqrt z\arcsin(\sqrt z)}{\sqrt{1-z}}
  \sum_{\ell=0}^k\frac{(\frac12)_\ell}{\ell!}\frac1{(1-z)^\ell}\\
  +  \sum_{\ell=1}^k\frac1\ell\left(\frac{(k+\frac32-\ell)_\ell}
    {(k+1-\ell)_\ell}-\frac{(\frac12)_\ell}{\ell!}\right)\frac1{(1-z)^\ell}
  -H_k,
\end{multline*}
which can be checked by differentiation. We get
\begin{equation*}
  \begin{split}
    \G_{k+\frac12,-\frac12}(\cos(2\theta))=&
    \frac1{2(k+1)}\cos(\theta)\left(\frac\pi2-\theta\right)
    \sum_{\ell=0}^k\frac{(\frac12)_\ell}{\ell!}
    \frac1{\sin(\theta)^{2\ell+1}}\\
    &+
    \frac1{4(k+1)}\sum_{\ell=1}^k\frac1\ell\left(\frac{(k+\frac32-\ell)_\ell}
      {(k+1-\ell)_\ell}
      -\frac{(\frac12)_\ell}{\ell!}\right)\frac1{\sin(\theta)^{2\ell}}\\
    &- \frac1{2(k+1)}H_k-\frac1{4(k+1)^2}.
  \end{split}
\end{equation*}
\end{proof}

Putting everything together we derive the asymptotic main term of the
Green function
\begin{equation}\label{eq:Green-asymp}
  \mathcal{G}_{\alpha,\beta}(\cos(2\theta))=
  \frac{\Gamma(\alpha)\Gamma(\beta+1)}{4\Gamma(\alpha+\beta+2)}
  \frac1{\sin(\theta)^{2\alpha}}+
  \mathcal{O}\left(\frac1{\sin(\theta)^{2\alpha-1}}\right),
\end{equation}
for $\alpha > 0$ and 
\begin{equation}
  \label{eq:RP2Green-asymp}
  \mathcal{G}_{0,-\frac{1}{2}}(\cos(2\theta))=
  -\log(\sin(\theta))+  \mathcal{O}\left(1\right),
\end{equation}
for $\Omega = \mathbb{RP}^2$.

\subsection{Minimizers of the Riesz, Green, and logarithmic Energies}\label{sec:MinimizersOfEnergies}

In this section, we prove that the uniform measure on $\Omega$ induced by the
Haar measure of the group acting on $\Omega$ minimizes the Riesz, Green
and logarithmic energies.

\begin{theorem}\label{thm:UniformMinEnergies}
The logarithmic energy $I_{K_0}$, the Green energy $I_G$, and the Riesz
  $s$-energies $I_{K_s}$, for $0 < s < 2 \alpha + 2$, are uniquely minimized by
  the uniform measure $\sigma$, with
\begin{equation}\label{eq:MinimalRieszEnergyCont}
I_{K_s}(\sigma) = \frac{\Gamma( \alpha + \beta + 2) \Gamma(\alpha + 1 - \frac{s}{2})}{\Gamma(\alpha + 1) \Gamma(\alpha + \beta + 2 - \frac{s}{2})},  
\end{equation}

\begin{equation}\label{eq:MinimalLogEnergyCont}
I_{K_0}(\sigma) = \frac{1}{2}( \psi( \alpha + \beta +2) - \psi( \alpha +1) ),  
\end{equation}
and
\begin{equation}\label{eq:MinimalGreenEnergyCont}
I_{G}(\sigma) = 0.  
\end{equation}
 Moreover, if $\{ \omega_N\}_{N=2}^{\infty}$ is
  a sequence of minimizers for the discrete energies $E_{K_0}$, $E_{K_s}$, or
  $E_G$, then the normalized counting measures 
  \begin{equation*}
\nu_{\omega_N} = \frac{1}{N}\sum_{j=1}^{N} \delta_{z_j},
\end{equation*}
converge weakly to $\sigma$.
\end{theorem}
We note that this has already been proven for Green energies in \cite[Theorem~1.1]{Beltran_Corral_CriadodelRey2019:GreenEnergy}, and the uniformity of minimizers for logarithmic and Riesz energies on $\mathbb{RP}^{d-1}$ and $\mathbb{CP}^{d-1}$ was proven in \cite[Theorem 4.1]{Chen_Hardin_Saff2021:search_tight_frames}.

In general, the minimal discrete and continuous energies of some kernel $K$ are related in the following way
\begin{theorem}[\!\!{\cite[Theorem 4.2.2]{BorodachovHardinSaff2019:EnergyRectifiableSets}}]\label{thm:RelateDiscreteContEnergies}
Let $K$ be a lower semi-conti\-nuous
  kernel on $\Omega$, and $\mu$ be a minimizer of $I_K$. Then for all $N \in \mathbb{N}\setminus\{1\}$,
\begin{equation*}
\mathcal{E}_K(N) \leq N(N-1) I_{K}(\mu),
\end{equation*}
and
\begin{equation*}
\lim_{N \rightarrow \infty} \frac{\mathcal{E}_K(N)}{N^2} = I_{K}(\mu).
\end{equation*}
If $\{ \omega_N\}_{N=2}^{\infty}$ is a sequence of optimal $K$-energy
configurations and $\nu$ is a weak$^*$ limit point of the normalized counting
measures $\nu_{\omega_N}$, then $\nu$ is an equilibrium measure of $I_K$.
\end{theorem}

Thus, if the continuous energy is minimized by the uniform measure, we should expect uniformly distributed discrete minimizers. That the uniform measure (uniquely) minimizes the energies follows from the strict positive definiteness of the kernels, which we show below.

\begin{definition}\label{def:PosDef}
  We call a kernel \emph{conditionally positive definite} if for all
  $\nu \in \mathcal{Z}(\Omega)$, for which the energy is well defined,
  $I_K(\nu) \geq 0$. We call a kernel \emph{positive definite} if for all
  $\nu \in \mathcal{M}(\Omega)$ for which the energy is well defined,
  $I_K(\nu) \geq 0$.  We call a kernel \emph{strictly positive definite} or
  \emph{conditionally strictly positive definite} if it is positive definite
  or conditionally positive definite, respectively, and $I_K(\nu) = 0$ only if
  $\nu = 0$.
\end{definition}

\begin{theorem}[see {\cite[Theorem~4.2.7]%
    {BorodachovHardinSaff2019:EnergyRectifiableSets}}]
  \label{thm:StrictPosDefUniqueMin}
  If $K$ is conditionally strict\-ly positive definite, then $I_K$ has a unique
  minimizer in $\mathbb{P}(\Omega)$.
\end{theorem}

We note that the proofs for Theorems \ref{thm:RelateDiscreteContEnergies} and
\ref{thm:StrictPosDefUniqueMin} given in
\cite{BorodachovHardinSaff2019:EnergyRectifiableSets} are for compact subsets
of $\mathbb{R}^d$, but they can be generalized to compact metric spaces.

\begin{corollary}\label{cor:RotationInvariantMin}
If $K$ is conditionally strictly positive definite and isometry invariant, then $\sigma$ is the unique minimizer of $I_K$.
\end{corollary}

\begin{proof}
Suppose that $\mu$ minimizes $I_K$. Since $K$ is isometry invariant, we have for all $g \in G$,
\begin{align*}
I_K(\mu) & = \iint\limits_{\Omega\times\Omega} K(x,y) \,d \mu(x) \,d\mu(y) \\
& = \iint\limits_{\Omega\times\Omega} K(g(x),g(y)) \,d \mu(x) \,d\mu(y) \\
& = I_K( g_{\#}(\mu))
\end{align*}
where $g_{\#}(\mu)$ is the pushforward measure of $\mu$ under $g$. Since $\mu$ must be unique, it must be isometry invariant, giving us our claim.
\end{proof}

\begin{lemma}\label{lem:UniformConvergenceJacobi}
If $F \in C([-1,1])$ and $K(x,y) = F(\cos(2 \vartheta(x,y)))$ is conditionally positive definite, then
\begin{equation*}F(t) = \sum_{n=0}^{\infty} \widehat{F}(n) P_n^{(\alpha, \beta)}(t),\end{equation*}
where the series converges absolutely and uniformly on $[-1,1]$.
\end{lemma}
\begin{proof}
  We first observe that $K(x,y)+C$ is positive definite for a large enough
  constant $C$ by \cite[Theorem~2]{Bochner1941:Positive_Definite}. Thus we can
  assume without loss of generality that $K(x,y)$ is positive definite. Then
  the eigenfunctions of the operator
  $T_K:f\mapsto\int_\Omega K(x,y)f(y)\,d\sigma(y)$ are the eigenfunctions of
  the Laplace operator and the corresponding eigenvalues are
  positive. Then the claim follows from Mercer's theorem (see, for instance
  \cite{Ferreira_Menegatto2009:eigenvalues_integral_operators}).
\end{proof}

\begin{theorem}\label{thm:PosDefCoefficients}
If $F \in C([-1,1])$, then $K(x,y) = F( \cos(2 \vartheta(x,y)))$ is
\begin{enumerate}
\item\label{pt:a} conditionally positive definite if and only if
  $\widehat{F}(n) \geq 0$ for all $n \in \mathbb{N}$,
\item\label{pt:b} positive definite if and only if $\widehat{F}(n) \geq 0$ for
  all $n \in \mathbb{N}_0$,
\item\label{pt:c} conditionally strictly positive definite if and only if
  $\widehat{F}(n) > 0$ for all $n \in \mathbb{N}$,
\item\label{pt:d} strictly positive definite if and only if $\widehat{F}(n) > 0$
  for all $n \in \mathbb{N}_0$.
\end{enumerate}
\end{theorem}

\begin{proof}
  The forward direction for each of these follows from \cite[Theorem~2 and
  Section~III]{Bochner1941:Positive_Definite}, with a slight alteration for the
  strict case.

  Part \eqref{pt:b} then follows from
  \cite[Lemma~2]{Bochner1941:Positive_Definite} and that uniform limits of
  positive definite functions are positive definite.

Now, suppose that $\widehat{F}(n) > 0$ for all $n \in \mathbb{N}_0$. Then $K$ is positive definite, so by Lemma \ref{lem:UniformConvergenceJacobi}, the addition formula~\eqref{eq:addition}, and the density of $\{ Y_{n,k}: n \in \mathbb{N}_0, k \in \{1, ..., \dim(V_n)\}\}$ in $C(\Omega)$, we have for any $\mu \in \mathcal{M}(\Omega)$, not identically zero,
\begin{align*}
I_K(\mu) & = \sum_{n=0}^{\infty} \widehat{F}(n) \iint\limits_{\Omega\times\Omega} P_n^{(\alpha, \beta)}(\cos(2 \vartheta(x,y))) \,d\mu(x) \,d\mu(y) \\
& = \sum_{n=0}^{\infty} \sum_{k=1}^{m_n} \frac{\widehat{F}(n) P_n^{(\alpha, \beta)}(1)}{m_n} \iint\limits_{\Omega\times\Omega} Y_{n,k}(x) Y_{n,k}(y) \,d\mu(x) \,d\mu(y) \\
& = \sum_{n=0}^{\infty} \sum_{k=1}^{m_n} \frac{\widehat{F}(n) P_n^{(\alpha, \beta)}(1)}{m_n} \Big( \int_{\Omega} Y_{n,k}(x) \,d\mu(x) \Big)^2 > 0,
\end{align*}
proving \eqref{pt:d}. Parts \eqref{pt:a} and \eqref{pt:c} now follow from
\cite[Theorem~2]{Bochner1941:Positive_Definite}.
 
\end{proof}

\begin{theorem}\label{thm:PosDefRiesz}
  For $0 \leq s <D$, the Riesz kernel $K_{s}$ is strictly positive
  definite.
\end{theorem}
\begin{proof}

For $\varepsilon > 0$ and $s\geq0$, let
\begin{equation*}
K_{s, \varepsilon}(x,y) =\begin{cases} \Big( \varepsilon + \frac{1- \cos(2 \vartheta (x,y))}{2} \Big)^{-\frac{s}{2}} & \text{ for } s> 0 \\
- \frac{1}{2}\log \Big( \varepsilon + \frac{1- \cos(2 \vartheta (x,y))}{2} \Big) & \text{ for }s = 0
\end{cases}
\end{equation*} 
and
\begin{equation*}
  F_{s, \varepsilon}(t) = \begin{cases}
    \Big( \varepsilon + \frac{1- t}{2} \Big)^{-\frac{s}{2}} & \text{ for }s> 0 \\
- \frac{1}{2}\log \Big( \varepsilon + \frac{1- t}{2} \Big) & \text{ for }s = 0
\end{cases}.
\end{equation*}

Note that for all $x,y \in \Omega$, $K_{s, \varepsilon}(x,y)$ is strictly decreasing in $\varepsilon$, and $\lim_{\varepsilon \rightarrow 0} K_{s, \varepsilon}(x,y) = K_s(x,y)$. Now, suppose that $\mu \in \mathcal{M}(\Omega)$ (not identically zero) such that $I_{K_s}(\mu)$ is well defined. There exists $\mu^+, \mu^- \in \mathcal{B}(\Omega)$ such that $\mu = \mu^+ - \mu^-$, and so, by the Monotone Convergence Theorem, we have
\begin{align*}
I_{K_s}(\mu) & = I_{K_s}(\mu^+) - 2 \iint\limits_{\Omega \times \Omega} K_s(x,y) d\mu^+(x) d\mu^-(y) + I_{K_s}(\mu^-) \\
& = \lim_{\varepsilon \rightarrow 0} \left(I_{K_{s, \varepsilon}}(\mu^+)
  -2 \iint\limits_{\Omega \times \Omega} K_{s,\varepsilon}(x,y) d\mu^+(x) d\mu^-(y) + I_{K_{s, \varepsilon}}(\mu^-) \right)\\
& = \lim_{\varepsilon \rightarrow 0} I_{K_{s, \varepsilon}}(\mu).
\end{align*}

We now show that $I_{K_{s, \varepsilon}}(\mu)$ is positive and strictly decreasing as a function of $\varepsilon$, for all $s \geq 0$.

For $\varepsilon > 0$,
\begin{equation}\label{eq:RieszEpsilonSeries}
\Big( \varepsilon + \frac{1- t}{2} \Big)^{-\frac{s}{2}} = 
\Big( \varepsilon + 1 \Big)^{-\frac{s}{2}}  \sum_{k=0}^{\infty} \binom{ k+\frac{s}{2}-1}{k}  \Big( \frac{t+1}{2( \varepsilon +1)} \Big)^k
\end{equation}
and
\begin{equation}\label{eq:LogEpsilonSeries}
- \log \Big( \varepsilon + \frac{1- t}{2} \Big) = 
- \log \Big( \varepsilon + 1 \Big) +  \sum_{k=1}^{\infty} \frac1k
\Big( \frac{t+1}{2( \varepsilon +1)} \Big)^k,
\end{equation}
with the series converging uniformly on $[-1,1]$.

The polynomials $(t+1)^k$ can be expressed as linear combinations of
Jacobi polynomials; the coefficients are given by
\begin{multline*}
  \frac{m_n}{P_n^{(\alpha,\beta)}(1)^2}\int_{0}^{\frac{\pi}{2}}(1+\cos(2\theta))^k
  P_n^{(\alpha,\beta)}(\cos(2\theta))\,
    d\nu^{\alpha,\beta}(\theta)\\
    =\frac{m_n}{P_n^{(\alpha,\beta)}(1)^2}\binom kn
  \frac{2^k(\alpha+1)_n(\beta+1)_k}{(\alpha+\beta+2)_{n+k}}.
\end{multline*}
Together with the series expansions \eqref{eq:RieszEpsilonSeries} and
\eqref{eq:LogEpsilonSeries} this shows that $\widehat{F_{s,\varepsilon}}(n)>0$
 for all $s \geq 0, n\in\mathbb{N}_0$, for $\varepsilon$ sufficiently small.

These coefficients are positive, meaning that $K_{s, \varepsilon}$ is strictly positive definite, for $\varepsilon$ sufficiently small, by Theorem \ref{thm:PosDefCoefficients}. Since $\mu \neq 0$, we know (from the density of $\operatorname{span}( \{ Y_{n,k}: n \in \mathbb{N}_0, 1 \leq k \leq \dim(V_n)\})$ in $C(\Omega)$) that for some $m \in \mathbb{N}_0$, we must have
\begin{equation*} \iint\limits_{\Omega\times\Omega} P_m^{(\alpha, \beta)}(\cos(2 \vartheta(x,y))) \,d\mu(x) \,d\mu(y) > 0.\end{equation*}
Since the coefficients $\widehat{F_{s, \varepsilon}}(n)$, for $n \in \mathbb{N}_0$, are strictly decreasing in $\varepsilon$ and positive, the positive definiteness of the Jacobi polynomials and Lemma \ref{lem:UniformConvergenceJacobi} gives us, for $0 \leq s < D$,
\begin{align*}
I_{K_s}(\mu) & = \lim_{\varepsilon \rightarrow 0} I_{K_{s, \varepsilon}}(\mu) \\
&  = \lim_{\varepsilon \rightarrow 0} \sum_{n=0}^{\infty} \widehat{F_{s, \varepsilon}}(n) \iint\limits_{\Omega\times\Omega} P_n^{(\alpha, \beta)}(\cos(2 \vartheta(x,y)))\, d\mu(x)\, d\mu(y) \\
& \geq \lim_{\varepsilon \rightarrow 0} \widehat{F_{s, \varepsilon}}(m) \iint\limits_{\Omega\times\Omega} P_m^{(\alpha, \beta)}(\cos(2 \vartheta(x,y)))\, d\mu(x)\, d\mu(y) \\
& > 0.
\end{align*}
This proves conditional strict positive definiteness.
\end{proof}

\begin{theorem}[see {\cite[Proposition~3.14]{Beltran_Corral_CriadodelRey2019:GreenEnergy}}]\label{thm:GreenPosDef}
  The Green function $G(x,y)$ is conditionally strictly positive definite.
\end{theorem}

\begin{proof}[Proof of Theorem \ref{thm:UniformMinEnergies}]

The first part of our claim follows from Theorems \ref{thm:PosDefRiesz} and \ref{thm:GreenPosDef} and Corollary \ref{cor:RotationInvariantMin}. Equations \eqref{eq:MinimalRieszEnergyCont}, \eqref{eq:MinimalLogEnergyCont}, and \eqref{eq:MinimalGreenEnergyCont} follow from direct computation or our assumptions on $G$.

By Theorem \ref{thm:RelateDiscreteContEnergies}, any weak$^*$ limit point of
$\{\nu_{\omega_N}: N \geq 2\}$ must be $\sigma$. We know, by the Banach-Alaoglu
Theorem, that $\mathbb{P}(\Omega)$ is weak$^*$ compact. Thus the sequence
$(\nu_{\omega_N})$ has a limit point, which has to be $\sigma$ by
Theorem~\ref{thm:RelateDiscreteContEnergies}. Thus the sequence
$(\nu_{\omega_N})$ converges to $\sigma$ and the second part of our claim now
follows.
\end{proof}

\subsection{The heat kernel}\label{sec:heat-kernel}
Using the theory of the heat
kernel for a compact Riemannian manifold (see, for example,
\cite{Rosenberg1997:Laplace_Riemannian_Manifolds}) we obtain a lower bound for
the Green function $G(x,y)$ on each of the (compact) projective spaces.

The \emph{heat kernel} on $\Omega$ is the unique function
$H_t(x,y):=H(t,x,y)\in C^{\infty}(\mathbb{R}^{+}\times \Omega\times \Omega)$
satisfying
\begin{align}
\Laplace_x H(t,x,y)+\frac{\partial}{\partial t}H(t,x,y) = 0\label{eq:heat-de}\\
\lim_{t\to 0^{+}}\int_{\Omega} H(t,x,y)f(y)\,d\sigma(y)=f(x)\label{eq:heat-limit}
\end{align}
for each $f\in C^{\infty}(\Omega),$ where $\Laplace_x=\Laplace_\Omega$ is the
Laplace operator in the variable $x$.

Similar to the case of the Green function, the spectral theorem and the
addition formula~\eqref{eq:addition} together imply that $H_t(x,y)$ has a series expansion in terms
of the Jacobi polynomials over $L^2(\Omega,\sigma)$:
\begin{equation*}
  \begin{split}
    H_t(x,y)=&\sum_{n=0}^{\infty}e^{-4n(n+\alpha+\beta+1)t}\frac{2n+\alpha+\beta+1}
    {(\alpha+\beta+1)}\\
    &\times \frac{(\alpha+\beta+1)_n(\alpha+1)_n}{n!(\beta+1)_n}
    P_n^{(\alpha,\beta)}(\cos(2\theta(x,y))).
  \end{split}
\end{equation*}

Contrasting with the formal expansion for $G(x,y)$, the series expansion of
$H_t(x,y)$ is uniformly convergent in $x,y\in \Omega$ for all $t>0$.

Integrating $(1-H_t(x,y))$ with respect to $t$, we arrive at the kernel 
\begin{equation}\label{eq:GreenHeatKernel}
  \begin{split}
    G_t(x,y)=&\sum_{n=1}^{\infty}e^{-4n(n+\alpha+\beta+1)t}
    \frac{2n + \alpha + \beta + 1}{4n(n+\alpha+\beta+1)}\\
    &\times \frac{(\alpha+\beta+2)_{n-1}(\alpha+1)_n}{(\beta+1)_nn!}
    P_n^{(\alpha,\beta)}(\cos(2\theta(x,y))),
  \end{split}
\end{equation}
which we may use to provide an explicit lower bound on the Green function.

\begin{lemma}\label{lem:ElkiesLemma} For all $t>0$ and $x\neq y$ we have 
\begin{equation}\label{eq:ElkiesIneq}
    G(x,y)\geq G_t(x,y)-t.
\end{equation}
\end{lemma}

This approximation is originally due to Elkies, and Lang presents a proof in
\cite[Lemma~5.2]{Lang1988:Introduction_to_Arakelov_Theory}. We shall also provide a proof, which uses the non-negativity of the heat kernel. 

\begin{theorem}\label{thm:PosHeatKernel}
For all $x,y\in \Omega$ and $t>0,$
\begin{align*}
    H_t(x,y)\geq0.
\end{align*}
\end{theorem}
Non-negativity follows from the fact that
$\frac{\partial}{\partial t}+\Laplace_{x}$ is a parabolic differential
operator, thus satisfying a \textit{strong maximum principle} (see
\cite[Chapter~3, Theorem~3]{Protter_Weinberger1984:Maximum_Principle},
\cite[page~152]{Lang1988:Introduction_to_Arakelov_Theory}.)

\begin{proof}[Proof of Lemma~\ref{lem:ElkiesLemma}]
  We observe that
  \begin{equation*}
    G_t(x,y)=\int_{\Omega}G(x,z)H_t(z,y)\,d\sigma(z).
  \end{equation*}
 This function is defined for all $(x,y)\in\Omega^2$ for $t>0$ by the
 integrability of $G(x,y)$. Furthermore, for all $x\neq y$ we have
 \begin{align*}
   \Laplace_x G_t(x,y)+\frac\partial{\partial t}G_t(x,y)&=0\\
   \lim_{t\to0}G_t(x,y)&=G(x,y)
 \end{align*}
 by \eqref{eq:heat-de}, \eqref{eq:heat-limit}, the integrability of $G$, and the continuity of $G$
 for $x\neq y$. 

 Then, by the defining property of $G$, we have
 \begin{align*}
   \Laplace_x G_t(x,y)=H_t(x,y)-1.
 \end{align*}
 From this, and Theorem \ref{thm:PosHeatKernel}, we derive, for $t_1>t_0>0$,
 \begin{equation*}
   -(t_1-t_0)\leq
   \int_{t_0}^{t_1}\left(H_t(x,y)-1\right)\,dt=
   G_{t_0}(x,y)-G_{t_1}(x,y).
 \end{equation*}
 Taking the limit $t_0\to0$ and using \eqref{eq:heat-limit}, we obtain
 \eqref{eq:ElkiesIneq}.
\end{proof}

In Section \ref{sec:lower-bounds-green}, we use this result to obtain
lower estimates for the Green energy on each of the projective spaces.


\section{Rotation invariant determinantal point
  processes on two-point  homogeneous spaces}
\label{sec:rotat-invar-determ}

We denote as $\X$ a (simple) random point process in the space
$\Omega$. To describe the process we specify the random variable
$\X(F)$ counting the number of points of the process in $F$, for all
Borel sets $F \subset \Omega$. A process is called \emph{simple}, if for
any $p\in\Omega$ we have $\X(\{p\})\leq1$, almost surely.

The joint intensities $\rho(x_{1},\ldots, x_{k})$ are functions defined in
$\Omega$ such that for any family of mutually disjoint subsets
$F_1, \ldots, F_{k} \subset\Omega$
\begin{equation*}
  \mathbb{E}[\X(F_1)\cdots\X(F_k)]=
  \idotsint\limits_{F_{1}\times \cdots \times F_{k}} \rho(x_{1},\ldots, x_{k})\,
  d\sigma(x_{1})\cdots \,d\sigma(x_{k}),
\end{equation*}
and we assume that $\rho(x_{1},\ldots, x_{k})=0$, when $x_{i}=x_{j}$ for
$i\neq j$.

\begin{definition}
  A random point process (see, e.g.,
  \cite[Chapter~4]{KrishnapurPeresBenHoughVirag_Zeros}) is called \emph{determinantal}
  with kernel $\K: \Omega \times \Omega \rightarrow \mathbb{C} $ if it is
  simple and the joint intensities with respect to a background measure
  $\sigma$ are given by
\begin{align*}
\rho(x_{1},\ldots, x_{k})=\det(\K(x_{i},x_{j}))_{1\leq i,j\leq k},
\end{align*} 
for every $k\geq 1$ and $x_{1},\ldots, x_{k}\in  \Omega $. 
 \end{definition}

 In \cite{KrishnapurPeresBenHoughVirag_Zeros}, it is shown that a determinantal
 process samples exactly $N$ points if and only if it is associated to the
 projection of $L^{2}(\Omega, \sigma)$ to an $N$-dimensional subspace $H$. Let
 $\phi_{1},\ldots,\phi_{N}$ be an orthornormal basis of $H$, then the
 projection kernel is given by
 \begin{equation*}
\K_{H}(x,y) = \sum\limits_{k=1}^{N}
\phi_{k}(x)\overline{\phi_{k}(y)}.
\end{equation*}

By the Macchi--Soshnikov theorem (see, e.g., \cite[Theorem
4.5.5]{KrishnapurPeresBenHoughVirag_Zeros}) the projection kernel $\K$ defines
a determinantal point process.
 
By Proposition~\ref{prop:finite-invariant} the only finite-dimensional
$G$-invariant subspaces of $L^2(\Omega,\sigma)$ are finite
orthogonal sums of eigenspaces of $\Laplace_{\Omega}$. Thus it is natural to consider
the subspace
\begin{equation*}
  H=V_0\oplus\cdots \oplus V_k
\end{equation*}
and the corresponding projection kernel given by
\begin{align*}
  \K_n^{(\alpha,\beta)}(x,y)=
  \sum_{k=0}^{n}
   \sum_{m=1}^{m_k} Y_{k, m}(x) Y_{k,m}(y), \ \ \   x, 
   y\in \Omega .
 \end{align*}
 This defines a $G$-invariant determinantal point process.

Using the addition formula~\eqref{eq:addition} and \eqref{eq:Jacobi-sum} the
kernel $\K_n^{(\alpha,\beta)}(x,y)$ can be written in the form
\begin{equation}\label{eq:kernel}
  \begin{split}
    \K_n^{(\alpha,\beta)}(x,y)=&
    \sum_{k=0}^{n} \frac{m_k}{P_k^{(\alpha, \beta)}(1)}
    P_k^{(\alpha, \beta)}(\cos(2\vartheta(x,y)) \\
    =&\frac{(\alpha + \beta + 2)_n}{(\beta + 1)_n}
    P_n^{(\alpha + 1, \beta)}(\cos(2\vartheta(x,y)),\quad
    x,y\in \Omega.
  \end{split}
\end{equation}
Then, for these kernels we have that
  \begin{equation}\label{trace}
    \begin{split}
      N&=\tr(\K_n^{(\alpha,\beta)(x,x)})=\int\limits_{\Omega}
      \K_n^{(\alpha,\beta)}(x,x)\, d\sigma(x)=\K_n^{(\alpha,\beta)}(1)\\
      &=\sum_{k=0}^{n} m_k =\frac{(\alpha+\beta+2)_n(\alpha+2)_n}{(\beta+1)_nn!} \sim
  \frac{\Gamma(\beta+1)}{\Gamma(\alpha+\beta+2)}\frac{n^{2\alpha+2}}{\Gamma(\alpha+2)},
    \end{split}
\end{equation} 
which by the Macchi--Soshnikov theorem is the number of points sampled by the
determinantal process associated to the kernel $\K_n^{(\alpha,\beta)}$, almost
surely. We shall call these determinantal point processes \emph{harmonic
  ensembles}.

Determinantal point processes are very convenient probabilistic models for the
study of energy expressions due to the following theorem
(see e.g., 
\cite[Equation (1.2.2)]{KrishnapurPeresBenHoughVirag_Zeros}).
\begin{proposition}\label{propositionDeterminantalProcess}
  Let $\K(x,y)$ be a projection kernel with trace $N$ in $\Omega$, and let
  $\omega_{N}=\{x_{1},\ldots, x_{N} \} \subset \Omega$ be $N$ random points generated
  by the corresponding determinantal point process $\X$. Then, for any
  measurable $f: \Omega\times\Omega\rightarrow[0,\infty)$, we have
\begin{equation*}
  \begin{split}
    \mathbb{E}_{\X_{N}}
    &\left(\sum\limits_{k\neq j}f(x_{k}, x_{j})\right)  \\
    =&\iint\limits_{\Omega\times\Omega}\left(\K(x,x)
      \K(y,y)-|\K(x,y)|^{2} \right)f(x,y)\,d\sigma(x)\,d\sigma(y).
  \end{split}
\end{equation*}
\end{proposition}


\section{Bounds for the Green, Riesz, and Logarithmic Energies on
  Projective Spaces}
\label{sec:bounds-green-riesz}

One often studies random configurations of points to find upper estimates for
the minimal energy. However, no local repulsion occurs between i.i.d. random
points, meaning that sampled points may concentrate near one another, and so
the expected discrete energy, $N(N-1)I_K(\mu)$, is too coarse of an estimate
for a good upper bound. One can prevent this clumping of sampled points by
distributing one point in each part of a partition of $\Omega$ (i.e. jittered
sampling). Alternatively, one can also generate random point sets with local
repulsion built in, using determinantal point processes. Recently,
determinantal point processes have been used to find bounds on energies in
various symmetric spaces (see,
e.g. \cite{Beltran_Etayo2018:projective_ensemble_distribution,
  Beltran_Ferizovic2020:approximation_to_uniform,
  Beltran_Marzo_Ortega-Cerda2016:energy_discrepancy_rotationally,
  Alishahi_Zamani2015:spherical_ensemble_uniform,
  Beltran_Delgado_L.+2021:gegenbauer_point_processes,
  Beltran_Etayo2019:generalization_spherical_ensemble, Hirao2021:finite_frames,
  Marzo_Ortega-Cerda2018:expected_riesz_energy}

Here, we use both jittered sampling and the determinantal point processes
introduced in Section~\ref{sec:rotat-invar-determ} to compute the expectations
of the discrete Riesz, Green, and logarithmic energies under these
models. These, of course, provide upper bounds for the minimal energies. For
the lower bounds of these energies, we use linear programming.  As in the case
of the sphere, the next-order term in the upper and lower bounds obtained by
these ideas have the same orders of magnitude in terms of the number of points
$N$.

\subsection{Upper bounds using jittered sampling}
\label{sec:upper-bounds-jittered}

Since each projective space is a connected Ahlfors regular metric measure space
with finite measure, we may use the following (formulated for our context):
\begin{proposition}[\!\!{\cite[Theorem~2]{Gigante_Leopardi2017:diameter_bounded_equal}}]\label{prop: Projective Partition}
  For each projective space $\Omega$, there exist positive constants $c_1$ and
  $c_2$ such that for all $N$ sufficiently large, there is a partition of
  $\Omega$ into $N$ regions each of measure $\frac{1}{N}$, contained in a
  geodesic ball of radius $c_1 N^{-\frac{1}{D}}$, and containing a
  geodesic ball of radius $c_2 N^{-\frac{1}{D}}$.
\end{proposition}

\begin{proposition}
  For the projective space $\Omega$ and $0 \leq s <D$, there is some
  positive constant $c_{\Omega,s}$ such that for $N \in \mathbb{N}$
  sufficiently large,
\begin{equation*}
\mathcal{E}_{K_s}(N)  \leq \begin{cases}
N^2 I_{K_s}(\sigma) - c_{\Omega, s} N^{1 + \frac{s}{D}} & \text{ for }s > 0 \\
N^2 I_{K_s}(\sigma) - c_{\Omega, s} N \log(N) & \text{ for }s = 0
\end{cases}.
\end{equation*}
\end{proposition}

\begin{proof}
  Since $\sin ( \vartheta ) < \vartheta < 2 \sin( \vartheta)$ for $\vartheta$
  sufficiently small, Proposition \ref{prop: Projective Partition} gives us
  that there exists some positive constant $c_{3}$ such that for $N$
  sufficiently large, there is a partition of $\Omega$ into $N$ regions, $D_1$,
  \ldots, $D_N$, each of measure $\frac{1}{N}$ and contained in a chordal ball
  of radius $c_3 N^{- \frac{1}{2 \alpha + 2}}$.

  Letting $ d \sigma_j(x) := N \mathbf{1}_{D_j} d \sigma(x)$, we have for
  $0 < s < D$,
\begin{align*}
&\mathcal{E}_{K_s}(N)  \leq \int_{\Omega} \cdots \int_{\Omega} \sum_{i \neq j} K_s( z_i, z_j)\,d \sigma_1(z_1) \cdots d \sigma_N(z_N) \\
& = N^2 \sum_{i \neq j} ~\iint\limits_{D_i\times D_j} K_s( z_i, z_j) \,d \sigma(z_i)\, d \sigma(z_j) \\
& = N^2 \Bigg(~\iint\limits_{\Omega\times\Omega} K_s (x,y) \,d \sigma(x)\, d \sigma(y) - \sum_{j=1}^{N} ~\iint\limits_{D_j\times D_j} K_s( x, y) \,d \sigma_j(x)\,
d \sigma_j(y)\Bigg)\\
& \leq N^2 I_{K_s}(\sigma) - \sum_{j=1}^{N} \frac{1}{\operatorname{diam}(D_j)^s} 
 \leq N^2 I_{K_s}(\sigma) - \frac{1}{(2 c_3)^{s}} N^{1 + \frac{s}{D}} .
\end{align*}
The logarithmic case works similarly. 
\end{proof}

We immediately have the following corollary

\begin{corollary}
For the projective space $\Omega$ there is some positive constant $c_{\Omega,G}$ such that for $N \in \mathbb{N}$ sufficiently large,
\begin{equation*}
\mathcal{E}_{G}(N)  \leq \begin{cases}  
- c_{\Omega, G} N \log(N) & \text{ for }\Omega = \mathbb{RP}^2\\
- c_{\Omega, G} N^{2- \frac{2}{D}} & \text{ for }\Omega \neq \mathbb{RP}^2
\end{cases}.
\end{equation*}
\end{corollary}

\subsection{Upper bounds using determinantal point processes}
\label{sec:upper-bounds-riesz}
In this section we study the expectation of $E_{K_s}$ and $E_G$ under the harmonic
ensemble given by the projection kernel \eqref{eq:kernel}. For the Riesz and logarithmic energies, this amounts to the
computation of integrals of the form
\begin{equation*}
  \mathbb{E}_{\X_N} \big[ E_{K_s} \big]=\iint\limits_{\Omega\times\Omega}
  \frac{\K_n^{(\alpha,\beta)}(1)^2-\K_n^{(\alpha,\beta)}(\cos(2\theta(x,y)))^2}
  {\sin(\theta(x,y))^s}
  \,d\sigma(x)\,d\sigma(y).
\end{equation*}
Using \eqref{eq:JacobiProjection} this simplifies to
\begin{equation*}
  \frac{(\alpha+\beta+2)_n^2}{(\beta+1)_n^2}\int_{0}^{\frac\pi2}
  \frac1{\sin(\theta)^s}
  \left(P_n^{(\alpha+1,\beta)}(1)^2-P_n^{(\alpha+1,\beta)}(\cos(2\theta))^2\right)
  \,d\nu^{\alpha,\beta}(\theta).
\end{equation*}
We first observe that the integral can be written as
\begin{equation*}
  \frac1{\gamma^{\alpha,\beta}}\int_0^{\frac\pi2}
  \frac{P_n^{(\alpha+1,\beta)}(1)^2-P_n^{(\alpha+1,\beta)}(\cos(2\theta))^2}
  {\sin(\theta)^2}\sin(\theta)^{2\alpha+3-s}\cos(\theta)^{2\beta+1}\,d\theta.
\end{equation*}
Now the limit of the quotient exists for $\theta\to0$, which shows that the
integral converges if $2\alpha+3-s>-1$. This gives
\begin{theorem}\label{thm:Determinantal Riesz Energy Finite}
  The expected Riesz $s$-energy $\mathbb{E}_{\X_N}\big[ E_{K_s} \big]$ is
  finite if and only if $s <D+2$.
\end{theorem}
Furthermore, using a similar reasoning the integral
\begin{equation}\label{eq:int-jacobi}
  \int_{0}^{\frac\pi2}\frac1{\sin(\theta)^s}P_n^{(\alpha+1,\beta)}(\cos(2\theta))^2
  \,d\nu^{\alpha,\beta}(\theta)
\end{equation}
converges, if and only if $s<D$. In order to derive the asymptotic
behaviour of these integrals for $n\to\infty$ we use the classical Hilb
approximation for the Jacobi polynomials (see
\cite[Theorem~8.21.12]{Szegoe1975:orthogonal_polynomials})
\begin{equation}
  \label{eq:jacobi-asymp}
  \begin{split}
    & \frac1{\binom{n+\alpha+1}n}
    P_n^{(\alpha+1,\beta)}(\cos(2\theta))\sin(\theta)^{\alpha+\frac32}
    \cos(\theta)^{\beta+\frac12} \\
    &=\Gamma(\alpha+2)\nt^{-\alpha-1}\sqrt\theta
    J_{\alpha+1}(2\nt\theta)+
    \begin{cases}
      \mathcal{O}\left(\theta^{\alpha+3}\right)
      &\text{for }0\leq\theta\leq\frac cn\\
      \mathcal{O}\left(\theta^{\frac12}n^{-\alpha-\frac52}\right)
      &\text{for }\frac cn\leq\theta\leq\frac\pi4
    \end{cases},
  \end{split}
\end{equation}
where $J_\alpha(x)$ denotes the Bessel function and $\nt=n+\frac12(\alpha+\beta+1)$. The $\mathcal{O}$-terms are uniform in
$\theta\in[0,\frac\pi4]$ (we use a weaker formulation here than what is known).

For studying the asymptotic behaviour of \eqref{eq:int-jacobi} we insert
\eqref{eq:jacobi-asymp} and obtain
\begin{align*}
  \int_0^{\frac\pi4}&\frac1{\sin(\theta)^s}P_n^{(\alpha+1,\beta)}(\cos(2\theta))^2
  \,d\nu^{\alpha,\beta}(\theta)\\
  =&\frac1{\gamma^{\alpha,\beta}}{\binom{n+\alpha+1}n}^2\Gamma(\alpha+2)^2
  \nt^{-2\alpha-2}
 \int_0^{\frac\pi4}\frac\theta{\sin(\theta)^{s+2}}J_{\alpha+1}(2\nt\theta)^2
 \,d\theta\\
 &+ \mathcal{O}\left(n^{2\alpha+2}\int_0^{\frac 1n}
   \theta^{2\alpha+\frac52-s}\,d\theta
 +n^{-2}\int_{\frac1n}^{\frac\pi4}\theta^{-\frac32-s}\,d\theta\right),
\end{align*}
where we have used the estimates
\begin{equation*}
  J_{\alpha+1}(x)=
  \begin{cases}
    \mathcal{O}(x^{\alpha+1})&\text{ for }x\to0\\
    \mathcal{O}(x^{-\frac12})&\text{ for }x\to\infty
  \end{cases}
\end{equation*}
for the first and the second integral in the $\mathcal{O}$-term,
respectively. We also estimated $\sin(\theta)$ trivially with $\theta$. The
error term then turns into $\mathcal{O}(n^{s-\frac32})$.

Thus we are left with the asymptotic evaluation of the integral
\begin{equation*}
  \int_0^{\frac\pi4}\frac\theta{\sin(\theta)^{s+2}}J_{\alpha+1}(2\nt\theta)^2
  \,d\theta.
\end{equation*}
We substitute $2\nt\theta=\tau$ and split the integral to obtain
\begin{multline*}
  (2\nt)^s\Biggl(\int_0^{\sqrt n}
  \tau^{-s-1}J_{\alpha+1}(\tau)^2\,d\tau+\mathcal{O}\left(\frac1n\right)\\
  +     \int_{\sqrt n}^{\frac{\nt\pi}2}\frac\tau{(2\nt)^{s+2}\sin(\tau/2\nt)^{s+2}}
  J_{\alpha+1}(\tau)^2
    \,d\tau\Biggr);
\end{multline*}
here we have replaced $\sin(\theta)$ with $\theta$ and controlled the error in
the range $\theta<1/\sqrt n$. The second integral can be estimated by
$\mathcal{O}(n^{-\frac{s+1}2})$. For the first integral we use
\cite[Section~3.8.5]{Magnus_Oberhettinger_Soni1966:formulas_theorems_special}
\begin{equation}\label{eq:weber-schafheitlin}
  \int_0^\infty \tau^{-s-1}J_{\alpha+1}(\tau)^2\,d\tau=
  \frac1{2^{s+1}}\frac{\Gamma(s+1)\Gamma(\alpha+1-\frac s2)}
  {\Gamma(\frac{s+1}2)^2\Gamma(\alpha+2+\frac s2)}
\end{equation}
with an error
\begin{equation*}
  \int_{\sqrt n}^\infty \tau^{-s-1}J_{\alpha+1}(\tau)^2\,d\tau=
  \mathcal{O}(n^{-\frac{s+1}2}).
\end{equation*}
For the remaining integral we  estimate
\begin{multline*}
  \int_{\frac\pi4}^{\frac\pi2}\frac1{\sin(\theta)^s}
  P_n^{(\alpha+1,\beta)}(\cos(2\theta))^2
  \,d\nu^{\alpha,\beta}(\theta)\\=
  \mathcal{O}\left(\int_0^{\frac\pi2}
    P_n^{(\alpha+1,\beta)}(\cos(2\theta))^2\,d\nu^{\alpha+1,\beta}(\theta)\right)=
  \mathcal{O}\left(\frac1n\right).
\end{multline*}
Putting everything together, we obtain
\begin{multline*}
  \int_0^{\frac\pi2}\frac1{\sin(\theta)^s}P_n^{(\alpha+1,\beta)}(\cos(2\theta))^2
\,d\nu^{\alpha,\beta}(\theta)\\
=\frac{{\binom{n+\alpha+1}n}^2}{2\gamma_{\alpha, \beta}}\!\!\!
~~\frac{\Gamma(s+1)\Gamma(\alpha+1-\frac s2)}
{\Gamma(\frac{s+2}2)^2\Gamma(\alpha+2+\frac s2)}
\nt^{s-2\alpha-2}\!+\mathcal{O}(n^{\frac{s-1}2}).
\end{multline*}
This gives
\begin{equation}
  \label{eq:energy-final}
  \begin{split}
    \mathbb{E}_{\X_N}&\big[ E_{K_s} \big]=
    \K_n^{(\alpha,\beta)}(1)^2\Gamma(\alpha+\beta+2)
    \Gamma\left(\alpha+1-\frac s2\right)\\
&\times    \left(\frac1{\Gamma(\alpha+\beta+2-\frac s2)}-
      \frac{\Gamma(s+1)\Gamma(\alpha+2)^2}
      {\Gamma(\frac{s+2}2)^2\Gamma(\alpha+2+\frac s2)\Gamma(\beta+1)}
      \nt^{s-2\alpha-2}\right)\\
    &+\mathcal{O}\left(n^{2\alpha+2+\frac{s-1}2}\right).
  \end{split}
\end{equation}

\begin{theorem}\label{thm:Riesz-expected}
  For $0<s<D$ the expected value of the Riesz energy satisfies
  \begin{equation*}
    \begin{split}
       \mathbb{E}_{\X_N}\big[ &E_{K_s} \big]= I_{K_s}(\sigma)\K_n^{(\alpha,\beta)}(1)^2\\-
      &\frac{\Gamma(s+1)\Gamma(\alpha+1-\frac
        s2)\Gamma(\beta+1)} {\Gamma(\frac s2+1)^2\Gamma(\alpha+2+\frac
        s2)\Gamma(\alpha+1)
        \Gamma(\alpha+\beta+2)}n^{s+D}
      +\mathcal{O}(n^{s+D-1}).
    \end{split}
  \end{equation*}
    In terms of the number of points $N$ this gives
\begin{equation}\label{eq:Riesz-expected-N}
  \begin{split}
    \mathcal{E}_{K_s}(N)  \leq& ~\mathbb{E}_{\X_N}\big[ E_{K_s} \big]\\
    =&I_{K_s}(\sigma)N^2 -\frac{\Gamma(s+1)\Gamma(\alpha+1-\frac s2)\Gamma(\beta+1)}
    {\Gamma(\frac s2+1)^2\Gamma(\alpha+2+\frac s2)\Gamma(\alpha+1)
      \Gamma(\alpha+\beta+2)}\\
    &\times
    \left(\frac{\Gamma(\alpha+\beta+2)\Gamma(\alpha+2)}
      {\Gamma(\beta+1)}\right)^{\frac{2\alpha+s+2}{2\alpha+2}}
    N^{1+\frac sD}
    +\mathcal{O}(N^{1 +\frac{s-1}{D}}).
  \end{split}
\end{equation}
\end{theorem}

For the limiting case $s=D$, we observe that the implicit constant in
the error term in \eqref{eq:energy-final} remains bounded for
$s\to D$. Thus we can take the limit $s\to D$ to obtain
\begin{theorem}
  The expected energy in the limiting case $s=D$ satisfies
  \begin{equation*}
    \begin{split}
      \mathbb{E}_{\X_N}\big[ E_{K_{D}} \big]
      =&\frac{\Gamma(\alpha+\beta+2)}{\Gamma(\beta+1)}
      \K_n^{(\alpha,\beta)}(1)^2\Bigl(2\log(n)\\
      &+\psi(2\alpha+3)-2\psi(\alpha+2)
        -\psi(\beta+1)\Bigr)+\mathcal{O}(n^{D-1}).
    \end{split}
  \end{equation*}
  In terms of the number of points $N$ this gives
  \begin{equation}
    \label{eq:limiting-energy-N}
    \begin{split}
      &\mathcal{E}_{K_{D}}(N) \leq \mathbb{E}_{\X_N} \big[ E_{K_{D}} \big]
      =\frac{\Gamma(\alpha+\beta+2)}
      {(\alpha+1)\Gamma(\beta+1)}\\
     & \quad\times N^2\Bigl(\log(N)
      +\log\left(\frac{\Gamma(\alpha+2)\Gamma(\alpha+\beta+2)}
        {\Gamma(\beta+2)}\right)\\
      &\quad+2(\alpha+1)\left(\psi(2\alpha+3)-2\psi(\alpha+2)
        -\psi(\beta+1)\right)\Bigr)+\mathcal{O}(N^{2-\frac1{D}}).
    \end{split}
  \end{equation}
\end{theorem}

For the logarithmic energy we follow the same line of reasoning as above. We
first compute the integral
\begin{equation*}
  \int_0^{\frac\pi2}\log\left(\frac1{\sin(\theta)}\right)
  \,d\nu^{\alpha,\beta}(\theta)=\frac12\left(\psi(\alpha+\beta+2)-
    \psi(\alpha+1)\right),
\end{equation*}
which is done by computing
\begin{equation*}
  \left.\frac{\partial}{\partial s}\int_0^{\frac\pi2}\sin(\theta)^{-s}
  \,d\nu^{\alpha,\beta}(\theta)\right|_{s=0}.
\end{equation*}
Then we derive the asymptotic behaviour of
\begin{equation*}
  \int_0^{\frac\pi2}\log\left(\frac1{\sin(\theta)}\right)
  P_n^{(\alpha+1,\beta)}(\cos(2\theta))^2
  \,d\nu^{\alpha,\beta}(\theta)
\end{equation*}
by using \eqref{eq:jacobi-asymp} again. This gives
\begin{align*}
  &\int_0^{\frac\pi4}\log\left(\frac1{\sin(\theta)}\right)
  P_n^{(\alpha+1,\beta)}(\cos(2\theta))^2
  \,d\nu^{\alpha,\beta}(\theta)\\=
  &\frac1{\gamma^{\alpha,\beta}}{\binom{n+\alpha+1}n}^2\Gamma(\alpha+2)^2
  \nt^{-2\alpha-2}\\
  \times&\int_0^{\frac\pi4}\log\left(\frac1{\sin(\theta)}\right)
  \frac\theta{\sin(\theta)^{2}}J_{\alpha+1}(2\nt\theta)^2
 \,d\theta\\+
 &\mathcal{O}\left(n^{2\alpha+2}
   \int_0^{\frac 1n}\log(\theta)\theta^{2\alpha+\frac52}\,d\theta
 +n^{-2}\int_{\frac1n}^{\frac\pi4}\log(\theta)\theta^{-\frac32}\,d\theta\right).
\end{align*}
The error term becomes $\mathcal{O}(n^{-\frac32}\log(n))$. The remaining
integral is then treated as above, which gives
\begin{align*}
  &\int_0^{\frac\pi4}\log\left(\frac1{\sin(\theta)}\right)
  \frac\theta{\sin(\theta)^{2}}J_{\alpha+1}(2\nt\theta)^2
  \,d\theta\\
  =&\int_0^\infty\log\left(\frac{2\nt}{\tau}\right)J_{\alpha+1}(\tau)^2
  \,\frac{d\tau}\tau+\mathcal{O}\left(\frac{\log n}n\right)\\
  =&\frac1{2(\alpha+1)}\log(2n)\\
  -&\frac1{4(\alpha+1)^2}
  \left(2(\alpha+1)(\psi(\alpha+1)+\log(2))+1\right)
  +\mathcal{O}\left(\frac{\log n}n\right),
\end{align*}
where we have used
\begin{align*}
  \int_0^\infty \frac1t J_{\alpha+1}(t)^2\,dt&=
  \frac1{2(\alpha+1)}\\
  \int_0^\infty \frac{\log(t)}t J_{\alpha+1}(t)^2\,dt&=
  \frac1{2(\alpha+1)}\left(\psi(\alpha+1)+\log(2)\right)+\frac1{4(\alpha+1)^2},
\end{align*}
which can be derived from \eqref{eq:weber-schafheitlin}.

The remaining integral
\begin{equation*}
  \int_{\frac\pi4}^{\frac\pi2}\log\left(\frac1{\sin(\theta)}\right)
  P_n^{(\alpha+1,\beta)}(\cos(2\theta))^2\,d\nu^{(\alpha,\beta)}(\theta)
\end{equation*}
can be estimated by $\mathcal{O}(\frac1n)$ as above. Putting everything
together, this shows:
\begin{theorem}
\label{thm:log-energy-upper}
  The expected value of the logarithmic energy satisfies
  \begin{equation*}
    \begin{split}
      \mathbb{E}_{\X_N}\big[ &E_{K_0} \big]= I_{K_0}(\sigma)\K_n^{(\alpha,\beta)}(1)^2\\
      &-
      \frac{\Gamma(\beta+1)}{\Gamma(\alpha+2)\Gamma(\alpha+\beta+2)}
      n^{D}
      \left(\log(n)-\frac{\psi(\alpha+1)+\psi(\alpha+2)}{2}\right)\\
      &+ \mathcal{O}(n^{D-1}\log n).
    \end{split}
  \end{equation*}
  In terms of the number of points this gives
  \begin{equation}
    \label{eq:energy-log-N}
    \begin{split}
      \mathcal{E}_{K_0}(N) \leq& ~ \mathbb{E}_{\X_N}\big[ E_{K_0} \big] \\
      =&I_{K_0}(\sigma) N^2-\frac N{D}
      \Biggl(\log(N) +\log\left(\frac{\Gamma(\alpha+2)\Gamma(\alpha+\beta+1)}
        {\Gamma(\beta+1)}\right)\\
      & -(\alpha+1)(\psi(\alpha+1)+\psi(\alpha+2))\Biggr)+ \mathcal{O}(N^{1-\frac1{D}}\log N).
    \end{split}
  \end{equation}
\end{theorem}

\begin{remark}
  A different approach to the asymptotic study of integrals of the form
  \eqref{eq:int-jacobi} was used in
  \cite{Skriganov2019:stolarskys_invariance_principle_ii} and
  \cite{Brauchart_Grabner2022:weighted_gegenbauer}. There the integral was
  rewritten as a sum using connection formulas and the orthogonality
  relations. The generating functions of these expressions can then be
  expressed in terms of hypergeometric functions; in some special cases there
  is a closed form expression.
\end{remark}

By \eqref{eq:Green-asymp} and \eqref{eq:RP2Green-asymp} the Green energy is
closely related to the Riesz energy with exponent $s=D-2$ (or the
logarithmic energy for $D= 2$); the main difference between the Green
energy and this Riesz energy is that the integral of the Green function
vanishes. Thus we have
\begin{theorem}
If $\alpha > 0$, the expected value of the Green energy under the harmonic ensemble is given
  by
  \begin{equation*}
\mathbb{E}_{\X_N}\big[ E_{G} \big]=-\frac{\Gamma(\beta+1)^2}
    {4\alpha(2\alpha+1)\Gamma(\alpha+1)^2\Gamma(\alpha+\beta+2)^2}n^{2D-2}
+\mathcal{O}(n^{2D-3}).
\end{equation*}
In terms of the number of points $N$ this gives
\begin{equation}\label{eq:green-expected-N}
  \begin{split}
    &\mathcal{E}_{G}(N) \leq \mathbb{E}_{\X_N}\big[ E_{G} \big]\\
    &=-\frac{(\alpha+1)^2}{4\alpha(2\alpha+1)} \left(\frac
      {\Gamma(\alpha+2)\Gamma(\alpha+\beta+2)}
      {\Gamma(\beta+1)}\right)^{\frac1{\alpha+1}} N^{2-\frac{2}{D}}
    +\mathcal{O}(N^{2-\frac3{D}}).
  \end{split}
\end{equation}
If $\Omega = \mathbb{RP}^2$, then
\begin{equation}\label{eq:RP2green-expected-N}
  \begin{split}
    \mathcal{E}_{G}(N) & \leq \mathbb{E}_{\X_N}\big[ E_{G} \big]\\
    & = -\frac N{2}
      \Biggl(\log(N)+\log\left(\frac{3}
        {4}\right)  -1 + 2 \gamma \Biggr)+ \mathcal{O}(N^{\frac1{2}}\log N).
  \end{split}
\end{equation}
\end{theorem}

\subsection{Lower Bounds for Riesz Energy}
\label{subsec:lowerbounds}
In this section we develop lower bounds for the Riesz and logarithmic energies using a method
which originates from \cite{Wagner1990:means_distances_surface} for the case of the
sphere. This method has then been refined in
\cite{Brauchart2006:about_second_term} to provide matching asymptotic orders
for the upper and lower bounds; a general version, again for the sphere, is
given in \cite[Theorem~6.4.4]{BorodachovHardinSaff2019:EnergyRectifiableSets}.

We present a slightly simplified proof for the lower bounds. For this purpose
we need some lemmas.
\begin{lemma}\label{lem:Linear Prog Bound}
  If $K(x,y) = g( \cos(2 \vartheta(x,y)))$ is continuous and positive definite,
  then
\begin{equation*}
\mathcal{E}_K(N) \geq \widehat{g}(0)N^2 -  g(1)N.
\end{equation*}
\end{lemma}
We recall that a function $f:I\to\mathbb{R}$ is called completely
monotone on an interval $I$  if for all $n\geq 0$
\begin{equation*}
  \forall u\in I: (-1)^n f^{(n)}(u) \geq 0.
\end{equation*}
\begin{lemma}\label{lem:absolutely_monotone}
Let $f$ be completely monotone on $[0,\infty)$, and $g(1-2u) = f(u)$ for $u \in [0,1]$.  Then the coefficients $\widehat{g}(n)$ as given in \eqref{eq:JacobiCoeff} are all nonnegative for
  $n\geq 0$.
\end{lemma}
Taking $u = \sin(\vartheta(x,y))^2$, and using \eqref{eq:Chordal Metric Def}, we see that $g$ is a function of $\cos(2 \vartheta(x,y))$. Moreover $g(t)$ is absolutely monotonic on $[-1,1]$ (i.e. all derivatives of $g$ are nonnegative).  The proof is then essentially the same as the proof of
\cite[Theorem~5.2.14]{BorodachovHardinSaff2019:EnergyRectifiableSets}, changing
Gegenbauer polynomials and weights to Jacobi polynomials and weights.

Let $f$ be completely monotone on $(0,\infty)$. From Taylor's formula with the integral form of the remainder term we obtain, for $u>0$,
\begin{equation*}
  f(u)=\sum_{k=0}^n\frac{\delta^k}{k!}(-1)^kf^{(k)}(u+\delta)+
  \frac{(-1)^{n+1}}{n!}\int_0^\delta t^nf^{(n+1)}(u+t)\,dt.
\end{equation*}
This observation was the main ingredient in
\cite{Brauchart2006:about_second_term}. All
summands are positive and finite for $\delta>0$ and $u \in [0, \infty)$. Furthermore, all summands are positive
definite, taking $u=\sin(\vartheta(x,y))^2$ by Lemma~\ref{lem:absolutely_monotone}
and Theorem~\ref{thm:PosDefCoefficients}.

We apply Lemma \ref{lem:Linear Prog Bound} to the function
\begin{equation*}
  F_{n,\delta}(u)=
  \sum_{k=0}^n\frac{\delta^k}{k!}(-1)^kf^{(k)}(u+\delta) \leq f(u)-\frac{(-1)^{n+1}}{n!}\int_0^\delta
  t^n f^{(n+1)}(u+t)\,dt,
\end{equation*}
with the inequality being an equality for $u > 0$, to obtain
\begin{equation}\label{eq:lower-f}
  \begin{split}
    E_{K_f}(\omega_N) \geq& N^2 \Biggl(\int_0^{\frac\pi2}
      f(\sin(\theta)^2)\,d\nu^{(\alpha,\beta)}(\theta)\\
      &-\frac{(-1)^{n+1}}{n!}\int_0^\delta t^n
      \int_0^{\frac\pi2}f^{(n+1)}\left(t+\sin(\theta)^2\right)\,
      d\nu^{(\alpha,\beta)}(\theta)\,dt\Biggr)\\
      &-NF_{n,\delta}(0).
  \end{split}
\end{equation}
We now apply the above observations to the functions
\begin{equation*}
f_s(u) = \begin{cases}
u^{-s/2} &\text{for } s>0 \\
- \frac{1}{2} \log(u) &\text{for } s = 0.
\end{cases}
\end{equation*}
Then the derivatives are given by
\begin{equation*}
  f_s^{(k)}(u)=(-1)^k c_{s,k}f_{s+2k}(u),
\end{equation*}
with
\begin{equation*}
  c_{s,k}=
  \begin{cases}
    1&\text{for }k=0\\
    \left(\frac s2\right)_k&\text{for }s>0\quad\text{and }k>0\\
    \frac12(k-1)!&\text{for }s=0\quad\text{and }k>0.
  \end{cases}
\end{equation*}
Then \eqref{eq:lower-f} becomes
\begin{align*}
  E_{K_s}&(\omega_N)\\
  \geq& ~N^2\left(\int\limits_0^{\frac\pi2}\!\! f_s(\sin(\theta)^2)\,
    d\nu^{(\alpha,\beta)}(\theta)-\frac{c_{s,n+1}}{n!}\!
    \int\limits_0^\delta \!t^n\!\!\int\limits_0^{\frac\pi2}\!\!
    \frac{d\nu^{(\alpha,\beta)}(\theta)}{(t+\sin(\theta)^2)^{\frac s2+n+1}}
    \,dt\right)\\
  &-N\left(f_s(\delta)+\delta^{-\frac s2}\sum_{k=1}^n\frac{c_{s,k}}{k!}\right).
\end{align*}
The inner integral then computes as
\begin{equation*}
  \int\limits_0^{\frac\pi2}
  \frac{d\nu^{(\alpha,\beta)}(\theta)}{(t+\sin(\theta)^2)^{\frac s2+n+1}}=
  (1+t)^{-\frac s2-n-1}\Hypergeom21{\beta+1,\frac s2+n+1}{\alpha+\beta+2}
  {\frac1{1+t}}.
\end{equation*}
From standard transformations for hypergeometric functions
(see~\cite[Section~2.4.1]{Magnus_Oberhettinger_Soni1966:formulas_theorems_special})
we obtain the asymptotic equivalent as $t\to0$
\begin{multline*}
  (1+t)^{-\frac s2-n-1}\Hypergeom21{\beta+1,\frac s2+n+1}{\alpha+\beta+2}
  {\frac1{1+t}}\\\sim \frac{\Gamma(\alpha+\beta+2)\Gamma(\frac s2+n-\alpha)}
  {\Gamma(\beta+1)\Gamma(\frac s2+n+1)}t^{\alpha-\frac s2-n}
\end{multline*}
valid for $n>\alpha-\frac s2$. Thus we choose $n$ as the smallest integer with
this property and obtain as $\delta\to0$
\begin{equation*}
  \int_0^\delta t^n\int\limits_0^{\frac\pi2}
  \frac{d\nu^{(\alpha,\beta)}(\theta)}{(t+\sin(\theta)^2)^{\frac s2+n+1}}
  \,dt\sim
 \frac{\Gamma(\alpha+\beta+2)\Gamma(\frac s2+n-\alpha)\delta^{\alpha+1-\frac s2}}
  {\Gamma(\beta+1)\Gamma(\frac s2+n+1)(\alpha+1-\frac s2)}.
\end{equation*}
Choosing $\delta=N^{-\frac1{\alpha+1}}$ gives
\begin{theorem}\label{thm:lower-riesz}
  Let $0< s<D$, then there is a constant $C_{s,D}>0$ such that
  \begin{equation*}
    E_{K_s}(\omega_N)\geq
   I_{K_s}(\sigma) N^2-C_{s,D}N^{1+\frac s{D}}.
  \end{equation*}
  For $s=0$ there is a constant $C_{0,D}>0$ such that
  \begin{equation*}
    E_{K_0}(\omega_N)\geq
    I_{K_0}(\sigma) N^2
    -\frac1{D}N\log N-C_{0,D}N.
  \end{equation*}
\end{theorem}

\subsection{Lower bounds for the Green energy}
\label{sec:lower-bounds-green} 
In this section we compute lower estimates for Green energy on each of the projective spaces. For $\mathbb{RP}^2$ the lower bound follows immediately from the lower bound on the logarithmic energy.

\begin{theorem}\label{thm:lower-Green-RP2}
There exists some constant $C_G > 0$ such that the Green energy of every point configuration $\omega_N$ on $\mathbb{RP}^2$, with $N \geq 2$, satisfies
\begin{equation*}
E_G(\omega_N) \geq -\frac{1}{2}N \log(N) - C_G N.
\end{equation*}
\end{theorem}
\begin{proof}
This follows immediately from Theorem \ref{thm:lower-riesz} and the fact that
\begin{equation*}
G(x,y) = K_0(x,y) - I_{K_0}(\sigma)
\end{equation*}
on $\mathbb{RP}^2$.
\end{proof}

For the other spaces, we employ a method developed in
\cite[Chapter~VI, \S~5]{Lang1988:Introduction_to_Arakelov_Theory}, which makes use of Lemma \ref{lem:ElkiesLemma}.

\begin{theorem}\label{thm:green-LB}
If $\alpha = 1/2$ \textup{(}i.e. $\Omega = \mathbb{RP}^3$\textup{)} , the Green energy of any collection of distinct points $\{x_1, \ldots, x_N\} \subset \Omega$ is bounded below
  by
\begin{equation*}
E_G(\omega_N) \geq - \frac{3}{4} \pi^{\frac{1}{3}} N^{2-2/3} +\mathcal{O}(N\log(N)).
\end{equation*}
For $\alpha > 1/2$, 
\begin{equation*}
E_G(\omega_N) \geq - \frac{1+\alpha}{4\alpha}\left(\frac{\Gamma\left(\beta+1\right)}{\Gamma(\alpha+\beta+2)} \right)^{\frac{1}{\alpha+1}}
 N^{2-\frac{2}{D}}
+\mathcal{O}(N^{2-\frac{3}{D}}).
\end{equation*}
\end{theorem}

In order to prove Theorem \ref{thm:green-LB}, we need the following two
lemmas. The proof of the lemmas are given after the proof of Theorem
\ref{thm:green-LB}.

\begin{lemma}\label{lemma:approxGeneralC}
For $\delta, c, d \geq 0$, and as $t\to0$
\begin{align*}
\sum_{k=1}^{\infty}
  \left(ck+d \right)^{\delta}e^{-2(ck+d)^2t} =   \frac{(2t)^{-\frac{\delta+1}{2}} }{2c}\Gamma\left(\frac{\delta+1}{2}\right) + \mathcal{O}(t^{-\delta/2})
\end{align*}
\end{lemma}

\begin{lemma}\label{lemma:approxCneq1}
For $d \geq 0$, 
$\sum_{k=1}^{\infty}
  \left(2k+d \right)^{-1}e^{-2(2k+d)^2t} = \mathcal{O}(\log(t))$, as $t\to0$.
\end{lemma}

\begin{proof}[Proof of Theorem \ref{thm:green-LB}]
For some $t > 0$ and any collection of distinct points $\{x_1, \ldots, x_N\} \subset \Omega$, we get
\begin{align*}
  &\sum\limits_{j\neq i}^{N} G(x_{j},x_{i})+ N(N-1)2t \geq
   \sum\limits_{j\neq i}^{N} G_{2t}(x_{j},x_{i}).
\end{align*}
Moreover from \eqref{eq:GreenHeatKernel} it follows
\begin{align*}
   \sum\limits_{j\neq i}^{N} G_{2t}(x_{j},x_{i})=& \sum\limits_{j\neq i}^{N}
   \sum_{k=1}^\infty
\frac{e^{-2\lambda_{k} t}}{\lambda_{k}}  
  \sum_{\ell=1}^{m_k}Y_{k,\ell}(x_{j})Y_{k,\ell}(x_{i})
  \notag \\
   =&
   \sum_{k=1}^{\infty}
\frac{1}{\lambda_{k}}  
  \sum_{\ell=1}^{m_k}
  \left( \left| \sum\limits_{j=1}^{N} e^{-\lambda_{k} t} Y_{k,\ell}(x_{j})
    \right|^{2}\!\!\!-\!\!
 \sum\limits_{j=1}^{N} e^{-2\lambda_{k}t}|Y_{k,\ell}(x_{j})|^{2} 
 \right)  
 \notag \\
   \geq&
   -\sum_{k=1}^{\infty}
\frac{1}{\lambda_{k}}  
  \sum_{\ell=1}^{m_k}
 \sum\limits_{j=1}^{N} e^{-2\lambda_{k}t}|Y_{k,\ell}(x_{j})|^{2} 
 \notag \\
   =&
   -\sum\limits_{j=1}^{N} G_{2t}(x_{j},x_{j})
 =- N G_{2t}(x, x).
\end{align*}
For all $t\geq 0$ and  $x,y \in \Omega$ 
\begin{align*}
&G_{t}(x,y) =\sum_{k=1}^{\infty}\frac{m_{k}}{\lambda_k}
  \frac{P_k^{(\alpha,\beta)}(\cos(2 \vartheta(x,y)))}{P_k^{(\alpha,\beta)}(1)}e^{-\lambda_k t},
\end{align*}
therefore  $G_{2t}(x, x)$ is equal to
\begin{align*}
 \sum_{k=1}^{\infty} &
 \frac{2k + \alpha + \beta + 1}{4k(k+\alpha+\beta+1)}
\frac{(\alpha+\beta+2)_{k-1}(\alpha+1)_k}{(\beta+1)_kk!} e^{-8k(k+\alpha+\beta+1)t}
  \notag\\
=&
\frac{\gamma_{\alpha,\beta}}{2\Gamma(\alpha+1)^2} 
\notag \\
\times &  \sum_{k=1}^{\infty}
\frac{2k+\alpha+\beta+1}{k(k+\alpha+\beta+1)}
  \frac{\Gamma(k+\alpha+\beta+1)  \Gamma(k+\alpha+1)}{ \Gamma(k+1)\Gamma(k+\beta+1)}
  e^{-8k(k+\alpha+\beta+1)t}.
\end{align*}
Taking into account that
\begin{align*}
  \frac{2k+\alpha+\beta+1}{k(k+\alpha+\beta+1)}
  &\frac{\Gamma(k+\alpha+\beta+1)  \Gamma(k+\alpha+1)}{ \Gamma(k+1)\Gamma(k+\beta+1)}
   \\
 =~&
 2\left(k+\frac{\alpha+\beta+1}{2} \right)^{{2\alpha-1}} \left(1+\mathcal{O}\left(\frac{1}{k} \right) \right)
\end{align*}
we obtain
 for $0<t \ll 1$, $G_{2t}(x, x)$ equals
 \begin{align*}
 &\frac{\gamma_{\alpha,\beta}}{\Gamma(\alpha+1)^2}  \sum_{k=1}^{\infty}
  \left(k+\frac{\alpha+\beta+1}{2} \right)^{{2\alpha-1}}e^{-8k(k+\alpha+\beta+1)t} \left(1+\mathcal{O}\left(\frac{1}{k} \right) \right)\\
  &=\frac{\gamma_{\alpha,\beta}e^{2(\alpha+\beta+1)^2t}
}{2^{2\alpha-1}\Gamma(\alpha+1)^2}  \sum_{k=1}^{\infty} \frac{e^{-2(2x+\alpha+\beta+1)^{2}t}}{\left(2k+\alpha+\beta+1\right)^{{1-2\alpha}}}
  \left(1+\mathcal{O}\left(\frac{1}{k} \right) \right).
 \end{align*}

Applying Lemma \ref{lemma:approxGeneralC} for $\alpha > 1/2, \delta = 2\alpha-1, c=2$ and $d = \alpha+\beta+1$, we obtain
  
  \begin{align*}
  G_{2t}(x, x)
=& 
\frac{e^{2{(\alpha+\beta+1)^{2}t}} \gamma_{\alpha,\beta}}{\Gamma(\alpha+1)^2 2^{2\alpha-1}} 
  \left(\frac{\Gamma(\alpha)}{2^{\alpha+2}} t^{-\alpha} + \mathcal{O}(t^{-\alpha+1/2})\right) + \mathcal{O}(t^{-\alpha+1/2})
 \notag \\
=& 
\frac{\Gamma(\alpha)\gamma_{\alpha,\beta}}{\Gamma(\alpha+1)^2 2^{3\alpha-1}}\left(1 + \mathcal{O}(t)\right) \left(t^{-\alpha} + \mathcal{O}(t^{-\alpha+1/2})\right)
 \notag \\
=& 
\frac{\Gamma(\beta+1)}{\alpha\Gamma(\alpha+\beta+1)^2 2^{3\alpha+2}}t^{-\alpha} + \mathcal{O}(t^{-\alpha+1/2})
 \notag \notag.
\end{align*}
By choosing $t= \frac{1}{8}\left(\frac{\Gamma(\beta+1)}{\Gamma(\alpha+\beta+2)} \right)^{\frac{1}{\alpha+1}} N^{-\frac{1}{\alpha+1}}$ in order to obtain a maximal lower bound and applying \eqref{trace}, we get
\begin{equation*}
E_G(\omega_N) \geq - \frac{1+\alpha}{4\alpha}\left(\frac{\Gamma(\beta+1)}{\Gamma(\alpha+\beta+2)} \right)^{\frac{1}{\alpha+1}}
 N^{2-\frac{1}{\alpha+1}}
+\mathcal{O}(N^{2-\frac{3}{2(\alpha+1)}}).
\end{equation*}
Applying Lemma \ref{lemma:approxCneq1} for $\alpha = 1/2, \delta = 0, c=2$ and $d = \alpha+\beta+1$, we obtain
  
  \begin{align*}
  G_{2t}(x, x)
&=  \frac{\gamma_{\alpha,\beta}e^{2t}
}{\Gamma(\alpha+1)^2}  \sum_{k=1}^{\infty} \frac{e^{-2(2x+\alpha+\beta+1)^{2}t}}{\left(2k+\alpha+\beta+1\right)^{{1-2\alpha}}}
  \left(1+\mathcal{O}\left(\frac{1}{k} \right) \right)
\end{align*}
Due to Lemma \ref{lemma:approxCneq1}, 
$$\sum_{k=1}^{\infty} e^{-2(2x+\alpha+\beta+1)^{2}t}\mathcal{O}\left(\frac{1}{k} \right)  = \mathcal{O}(\log(t)).$$
Furthermore, applying Lemma \ref{lemma:approxGeneralC},
$$\sum_{k=1}^{\infty} e^{-2(2x+\alpha+\beta+1)^{2}t} = 2^{-3/2}t^{-1/2}\sqrt{\pi}+ \mathcal{O}(1).$$
 Hence, we obtain
 $$ G_{2t}(x, x) =  \frac{\gamma_{\alpha,\beta}
}{\Gamma(\alpha+1)^2}2^{3/2}t^{-1/2}\sqrt{\pi}+ \mathcal{O}(\log(t))  $$
By choosing $t= \frac{1}{8}\left(\frac{\Gamma(\beta+1)}{\Gamma(\alpha+\beta+2)} \right)^{\frac{1}{\alpha+1}} N^{-\frac{1}{\alpha+1}}$ with $\alpha = 1/2$ and $\beta = -1/2$, we get
\begin{equation*}
E_G(\omega_N) \geq - \frac{3}{4} \pi^{\frac{1}{3}} N^{2-2/3} +\mathcal{O}(N\log(N)).
\end{equation*}
\end{proof}

\begin{proof}[Proof of Lemma \ref{lemma:approxGeneralC}]
Since
\begin{align*}
\sum_{k=1}^{\infty}
  \left(ck+d \right)^{\delta}e^{-2(ck+d)^2t} - \int_{1}^{\infty}  \left(cx+d \right)^{\delta}e^{-2(cx+d)^2t} dx 
\end{align*}
is lower bounded by $ e^{-2(c+d)^2t}(c+d)^\delta - e^{-\delta/2}\left(\frac{\delta}{4t}\right)^{\delta/2}$ and upper bounded by $ e^{-\delta/2}\left(\frac{\delta}{4t}\right)^{\delta/2}$, we obtain
\begin{align*}
\sum_{k=1}^{\infty}  \left(ck+d \right)^{\delta}e^{-2(ck+d)^2t} &= \int_{1}^{\infty}  \left(cx+d \right)^{\delta}e^{-2(cx+d)^2t} dx + \mathcal{O}(t^{-\delta/2})\\
&= \int_{0}^{\infty}  \left(cx+d \right)^{\delta}e^{-2(cx+d)^2t} dx + \mathcal{O}(t^{-\delta/2}).
\end{align*}
In the last equation we used the fact that 
$$ 0 \leq \int_{0}^{1}  \left(cx+d \right)^{\delta}e^{-2(cx+d)^2t} dx \leq e^{-\delta/2}\left(\frac{\delta}{4t}\right)^{\delta/2}.$$
Substituting $y$ by $(cx+d)\sqrt{2t}$ lead to the following equation
\begin{align*}
\sum_{k=1}^{\infty}  &\left(ck+d \right)^{\delta}e^{-2(ck+d)^2t} \\
&= \int_{d\sqrt{2t}}^{\infty}  y^{\delta}e^{-y^2} (2t)^{-\delta/2}dy + \mathcal{O}(t^{-\delta/2})\\
&=\frac{1}{c}(2t)^{-\frac{\delta+1}{2}}\left(\int_{0}^{\infty}  y^{\delta}e^{-y^2} dy - \int_{0}^{d\sqrt{2t}}  y^{\delta}e^{-y^2} dy\right) + \mathcal{O}(t^{-\delta/2})\\
&= \frac{1}{2c}(2t)^{-\frac{\delta+1}{2}} \Gamma\left(\frac{\delta+1}{2}\right) + \mathcal{O}(t^{-\delta/2}).
\end{align*}

\end{proof}

\begin{proof}[Proof of Lemma \ref{lemma:approxCneq1}]
The following sum
\begin{align*}
\sum_{k=1}^{\infty}
  \left(2k+d \right)^{-1}e^{-2(2k+d)^2t} - \int_{1}^{\infty}  \left(2x+d \right)^{-1}e^{-2(2x+d)^2t} dx 
\end{align*}
is lower bounded by $ 0$ and upper bounded by $\frac{1}{2+d}e^{-2(2+d)^2t}$. Therefore, we obtain
\begin{align*}
\sum_{k=1}^{\infty}
  \left(2k+d \right)^{-1}e^{-2(2k+d)^2t} =  \int_{1}^{\infty}  \left(2x+d \right)^{-1}e^{-2(2x+d)^2t} dx  + \mathcal{O}(1).
\end{align*}
Let $y = (2x+d)\sqrt{2t}$, then
\begin{align*}
\int_{1}^{\infty}  \frac{e^{-2(2x+d)^2t}}{\left(2x+d \right)} dx &=  \int_{(2+d)\sqrt{2t}}^{\infty}  \frac{\sqrt{2t}}{y} e^{-y^2} (2\sqrt{2t})^{-1}dy  \\
  &=  \frac{1}{2}\int_{(2+d)\sqrt{2t}}^{\infty}  \frac{1}{y} e^{-y^2} dy  \\
  &= \frac{\Gamma(0, (2+d)^22t)}{4}\\
  &= \frac{1}{4}\left(-\gamma -\log((2+d)^22t) - \sum_{k=1}^\infty \frac{(-(2+d)^22t)^k}{k!k}\right).
\end{align*}
\end{proof}


\begin{ackno}
  This research was initiated during the workshop ``Minimal energy problems
  with Riesz potentials'' held at the American Institute of Mathematics in
  May~2021. We would like to thank Carlos Beltr\'{a}n, Dmitriy Bilyk, and Damir Ferizovi\'{c} for
  their helpful suggestions, and Carlos for inspiring us to pursue this topic
  at the AIM workshop.
\end{ackno}
\bibliographystyle{amsplain}
\bibliography{refs}
\end{document}